\documentclass{siamart250211}

\usepackage{lipsum}
\usepackage{amsfonts}
\usepackage{graphicx}
\usepackage{epstopdf}
\usepackage{algorithmic}
\ifpdf
  \DeclareGraphicsExtensions{.eps,.pdf,.png,.jpg}
\else
  \DeclareGraphicsExtensions{.eps}
\fi

\usepackage{amssymb}
\usepackage{tikz}
\usetikzlibrary{calc,math}
\usepackage{pgfplots}
\usepackage{enumerate}
\usetikzlibrary{external}
\usepgfplotslibrary{groupplots}
\usetikzlibrary{patterns}

\renewcommand{\d}{\,{\rm d}} 
\DeclareMathOperator{\diam}{diam} 
 
\DeclareMathOperator{\supp}{supp}

\newcommand{\bs}[1]{\boldsymbol{#1}}

\newcommand{\const}[1]{C_{\text{\rm#1}}}

\newcommand{\set}[2]{\big\{#1\,:\,#2\big\}}

\newcommand{\dual}[3][]{#1(#2\,,\,#3#1)}
\newcommand{\norm}[3][]{#1\|#2#1\|_{#3}}

\newcommand\refine{{\tt refine}}

\newcommand\A{\mathbb{A}}

\newcommand\N{\mathbb{N}}
\newcommand\R{\mathbb{R}}
\newcommand\T{\mathbb{T}}

\newcommand{\NN}{\mathcal{N}}

\newcommand{\PP}{\mathcal{P}}
\newcommand{\RR}{\mathcal{R}}
\renewcommand{\SS}{\mathcal{S}}
\newcommand{\UU}{\mathcal{U}}
\newcommand{\VV}{\mathcal{V}}
\newcommand{\jump}[1]{[#1]}
\newcommand{\level}{{\tt level}}
\newcommand\push{{\tt push}}
\newcommand{\kTp}{\kappa_T^+}
\newcommand{\kTm}{\kappa_T^-}
\newcommand{\kTTp}{\kappa_{T'}^+}
\newcommand\MM{\mathcal M}

\newcommand\TT{\mathcal T}

\newcommand{\SlopeTriangle}[6]
{
    \pgfplotsextra
    {
        \pgfkeysgetvalue{/pgfplots/xmin}{\xmin}
        \pgfkeysgetvalue{/pgfplots/xmax}{\xmax}
        \pgfkeysgetvalue{/pgfplots/ymin}{\ymin}
        \pgfkeysgetvalue{/pgfplots/ymax}{\ymax}

        \pgfmathsetmacro{\xArel}{#1}
        \pgfmathsetmacro{\yArel}{#3}
        \pgfmathsetmacro{\xBrel}{#1-#2}
        \pgfmathsetmacro{\yBrel}{\yArel}
        \pgfmathsetmacro{\xCrel}{\xArel}
        \pgfmathsetmacro{\lnxB}{\xmin*(1-(#1-#2))+\xmax*(#1-#2)} % in [xmin,xmax].
        \pgfmathsetmacro{\lnxA}{\xmin*(1-#1)+\xmax*#1} % in [xmin,xmax].
        \pgfmathsetmacro{\lnyA}{\ymin*(1-#3)+\ymax*#3} % in [ymin,ymax].
        \pgfmathsetmacro{\lnyC}{\lnyA+#4*(\lnxA-\lnxB)}
        \pgfmathsetmacro{\yCrel}{\lnyC-\ymin)/(\ymax-\ymin)} % THE IMPROVED EXPRESSION WITHOUT 'DIMENSION TOO LARGE' ERROR.

        \coordinate (A) at (rel axis cs:\xArel,\yArel);
        \coordinate (B) at (rel axis cs:\xBrel,\yBrel);
        \coordinate (C) at (rel axis cs:\xCrel,\yCrel);

        \draw[#6]   (A)-- node[anchor=north] {#5}
                    (B)--
                    (C)--
                    cycle;
    }
}

\newsiamremark{remark}{Remark}
\newsiamremark{hypothesis}{Hypothesis}
\crefname{hypothesis}{Hypothesis}{Hypotheses}
\newsiamthm{claim}{Claim}
\newsiamremark{fact}{Fact}
\crefname{fact}{Fact}{Facts}

\headers{Optimal rates of AFEM for unbounded domains}{Th\'eophile Chaumont-Frelet and Gregor Gantner}

\title{Optimal convergence rates \\ of an adaptive finite element method \\ for unbounded domains%
\thanks{GG acknowledges funding by the Deutsche Forschungsgemeinschaft (DFG, German Research Foundation) under Germany's Excellence Strategy -- EXC-2047/1 -- 390685813 and the individual research grant 545527047.
TCF acknowledges funding by the ANR JCJC project APOWA (research grant ANR-23-CE40-0019-01).
}}

\author{Th\'eophile Chaumont-Frelet\thanks{Inria Univ.~Lille and Laboratoire Paul Painlev\'e, 59655 Villeneuve-d’Ascq, France
  (\email{theophile.chaumont@inria.fr}).} 
\and Gregor Gantner\thanks{Institute for Numerical Simulation, University of Bonn, Friedrich-Hirzebruch-Allee~7, 53115 Bonn, Germany
  (\email{gantner@ins-uni.bonn.de}).}}

\usepackage{amsopn}

\ifpdf
\hypersetup{
  pdftitle={{Optimal convergence rates \\ of an adaptive finite element method \\ for unbounded domains}},
  pdfauthor={Th\'eophile Chaumont-Frelet and Gregor Gantner}
}
\fi

\newtheorem{algo}[theorem]{{\sc Algorithm}}

\begin{document}

% Make title
% \date{\today}
\maketitle

% Abstract
\begin{abstract}
We consider linear reaction-diffusion equations posed on unbounded domains,
and discretized by adaptive Lagrange finite elements. To obtain finite-dimensional
spaces, it is necessary to introduce a truncation boundary, whereby only a
bounded computational subdomain is meshed, leading to an approximation of
the solution by zero in the remainder of the domain. We propose a residual-based
error estimator that accounts for both the standard discretization error as
well as the effect of the truncation boundary. This estimator is shown to
be reliable and efficient under appropriate assumptions on the triangulation.
Based on this estimator, we devise an adaptive algorithm that automatically
refines the mesh and pushes the truncation boundary towards infinity. We prove
that this algorithm converges and even achieves optimal rates in terms
of the number of degrees of freedom. We finally provide numerical examples
illustrating our key theoretical findings.
\end{abstract}



% Contents
%%%%%%%%%%%%%%%%%%%%%%%%%%%%%%%%%%%%%%%%%%%%%%%%%%%%%%%%%%%%%%%
\section{Introduction}
%%%%%%%%%%%%%%%%%%%%%%%%%%%%%%%%%%%%%%%%%%%%%%%%%%%%%%%%%%%%%%%

We consider second-order PDE problems posed on unbounded domains
and their discretization by finite elements methods. We are in particular
interested in cases where the boundary is unbounded and where the
PDE involves heterogeneous coefficients that are not necessarily
constant outside a compact subset. This setting renders (coupling
with) boundary element methods, as considered in~\cite{affkmp13,feischl22,gr25,kq11,rt19,sayas09},
highly challenging, and we therefore resort to a domain truncation approach.

Specifically, we fix a domain $\Omega\subseteq\R^d$ with Lipschitz boundary
$\partial\Omega$. Throughout, we suppose that $\Omega$ is a polytop in the sense that
it can be covered by a (potentially infinite) conforming triangulation $\widehat \TT_0$
of compact simplices. Let $\kappa \in L^\infty(\Omega)$ be a given reaction term which is
uniformly positive. For $f\in L^2(\Omega)$, we consider the reaction-diffusion equation
\begin{align}
\label{eq_rd_intro}
\begin{split}
	\kappa^2 u - \Delta u &= f \quad \text{in }\Omega,
	\\
	u &= 0 \quad \text{on }\partial\Omega.
\end{split}
\end{align}
The Lax--Milgram lemma guarantees unique existence of a weak solution $u\in H_0^1(\Omega)$ characterized by 
\begin{align}
\label{eq_rd_weak_intro}
	a(u,v) := \int_\Omega \kappa^2 u v + \nabla u \cdot \nabla v \d x = \int_\Omega f v \d x 
	\quad \forall v \in H_0^1(\Omega).
\end{align}
As mentioned earlier, we stress that neither the domain $\Omega$
nor its boundary $\partial\Omega$ is assumed to be bounded.
For simplicity, we restrict our attention to the case where the diffusion matrix is the identity
and no advection term is present, but more general situations can be covered without difficulty.

Specifically, the developed theory applies to general second-order PDEs with symmetric
uniformly positive definite diffusion matrix in $L^\infty(\Omega)$ with piecewise $W^{1,\infty}$
regularity
(this is required to define the residual-based estimator) and lower-order terms
with coefficients in $L^\infty(\Omega)$. In fact, the only key assumption
is that the resulting bilinear form is coercive with respect to the
\emph{full} $H^1$-norm
$\norm{v}{H_{\alpha}^1(\Omega)} := (\alpha^2\norm{v}{\Omega}^2 + \norm{\nabla v}{\Omega}^2)^{1/2}$,
$v\in H_0^1(\Omega)$, for some $\alpha > 0$. Note that this norm is \emph{not} equivalent to
the (semi-)norm $\norm{\nabla v}{\Omega}$ on unbounded domains $\Omega$. 
The only noteworthy modification in this case
arises in the proof of some general quasi-orthogonality when advection terms are present, 
as these lead to a non-symmetric bilinear form  $a(\cdot,\cdot)$, and the Pythagoras theorem
can therefore no longer be applied; see Remark~\ref{rem:nonsymmetric} for details and a
workaround based on the analysis of \cite{feischl22}.

To approximate the exact weak solution $u\in H_0^1(\Omega)$, we employ continuous piecewise polynomials of degree $p\in\N$ with vanishing trace on finite triangulations: 
Let $\widehat\TT_H$ be a (possibly infinite) conforming triangulation of $\Omega$ and $\TT_H\subset \widehat \TT_H$ a finite subtriangulation with associated set $\Omega_H := {\rm int}(\bigcup\TT_H) \subset \Omega$, where ${\rm int}(\cdot)$ denotes the interior of a set.
Note that we do not assume that $\Omega_H$ is a Lipschitz domain. 
We denote by $\Gamma_H := \partial\Omega_H \setminus \partial\Omega$ the artificial boundary introduced by the discretization.
Identifying functions on $\Omega_H$ with functions on $\Omega$ via extension by zero, we have the inclusion 
\begin{align}
	\SS_0^p(\TT_H) := \set{v_H \in C^0(\overline\Omega_H)}{ v_H|_T \in \PP^p(T) \,\,\, \forall T \in \TT_H, v_H|_{\partial\Omega_H} = 0}
	\subseteq H_0^1(\Omega). 
\end{align}
Again, the Lax--Milgram lemma applies and guarantees unique existence of a discrete Galerkin approximation $u_H \in \SS_0^p(\TT_H)$ with
\begin{align}
\label{eq_fem_intro}
	a(u_H,v_H) = \int_{\Omega_H} f v_H \d x 
	\quad \forall v_H \in \SS_0^p(\TT_H). 
\end{align}

The goals of the present contribution are twofold. Specifically, we aim to (i)
rigorously estimate the discretization error incurred in~\eqref{eq_fem_intro},
including the effect of domain truncation, by means of an error estimator, and (ii) employ
this error estimator to steer adaptive mesh refinement and domain truncation.

Our first contribution is to show that the standard residual-based error indicators
\begin{equation*}
\eta_H^2(T)
:=
h_T^2 \norm{f - \kappa^2 u_H + \Delta u_H}{T}^2
+
h_T \norm{\jump{\partial_{\bs{n}} u_H}}{\partial T \cap \Omega}^2
\quad
\forall T \in \TT_H
\end{equation*}
is well suited to capture the discretization error, including the effect of domain truncation.
Note that the normal jumps across the artificial boundary are incorporated into the estimator.

In Theorem~\ref{thm:efficiency}, by slightly adapting standard proofs, we show
that this estimator is (locally) efficient under the requirement that
$\kappa_T^+ h_T$ is uniformly bounded for all $T\in\widehat\TT_H$,
where $\kappa_T^+$ is the essential supremum of $\kappa$
on $T$ and $h_T:=|T|^{1/d}$ is the element size of $T$.
This requirement is sharp for the standard residual-based estimator,
even in bounded domains. In that case, it can be relaxed
by modifying the estimator as in~\cite{verfurth_1998a}, but we refrain
from doing so here for simplicity.

We further prove in Theorem~\ref{thm:reliability} that the estimator is reliable under
the additional condition that $\kappa_T^- h_T$ is uniformly bounded from below
for elements $T \in \TT_H$ touching the truncation boundary $\Gamma_H$, where 
$\kappa_T^-$ is the essential infimum of $\kappa$ over $T$. This
is inspired by ideas recently introduced by one of the authors in~\cite{chaumontfrelet_2025b},
where an estimator based on local flux equilibration is proposed. As discussed
in~\cite{chaumontfrelet_2025b}, the requirement that elements are not refined
at the truncation boundary is natural, and can be viewed as an extended shape-regularity
requirement on the mesh.

The second key contribution of this work is to employ the error estimator in
an adaptive loop based on the standard D\"orfler's marking and a suitable \emph{ad hoc}
refinement strategy that extend the truncated domain when necessary. We show that
sequence of approximate solutions produced by this adaptive algorithm converges R-linearly 
and that optimal rates are achieved in suitable approximation classes. The
proofs rely on the axioms of adaptivity introduced in~\cite{cfpp14}.

Our abstract analysis is complemented by numerical results.
These examples precisely match the theoretical predictions,
whereby we indeed observe optimal convergence with expected
rates.

The remainder of this work is organized as follows.
In Section~\ref{sec:aposteriori}, we introduce the estimator and establish its key properties.
Section~\ref{sec:refinement} extends standard results on mesh refinement in bounded domains
to the unbounded case. In Section~\ref{sec:adaptivity}, we present the adaptive algorithm
and state our main convergence results, whose proofs are given in Sections~\ref{sec:axioms}
and~\ref{sec:proofs}. Finally, numerical results are presented in Section~\ref{sec:numerics}.

%%%%%%%%%%%%%%%%%%%%%%%%%%%%%%%%%%%%%%%%%%%%%%%%%%%%%%%%%%%%%%%
\section{A posteriori error estimation}\label{sec:aposteriori}
%%%%%%%%%%%%%%%%%%%%%%%%%%%%%%%%%%%%%%%%%%%%%%%%%%%%%%%%%%%%%%%

Let $\widehat\TT_H$ be a conforming triangulation of $\Omega$ that is uniformly shape regular, i.e., there exists a constant $\const{sha}>0$ such that 
\begin{align}\label{eq:shape_regularity}
	\diam(T)^d \le \const{sha} |T| \quad \forall T\in\widehat\TT_H,
\end{align}
Moreover, let $\TT_H\subset \widehat\TT_H$ be a finite subtriangulation. 
With the element size $h_T = |T|^{1/d}$ for all $T\in\widehat\TT_H$, we can then define the refinement indicators
\begin{subequations}\label{eq:estimator}
\begin{align}
\begin{split}
	\eta_H(T, v_H)^2 := h_T^2 \norm{f - \kappa^2 v_H + \Delta v_H}{T}^2 + h_T \norm{\jump{\partial_{\bs{n}} v_H}}{\partial T \cap \Omega}^2
	\quad \forall T \in \TT_H, v_H \in \SS^p_0(\TT_H),
\end{split}
\end{align}
where, as usual, $\jump{\cdot}$ denotes the jump of across element interfaces in $\widehat\TT_H$ and $\partial_{\bs{n}}$ is the normal derivative. 
Based on these local contributions, we define
\begin{align}
	\eta_H(v_H) := \eta_H(\TT_H,v_H) \quad \text{with} \quad\eta_H(\UU_H, v_H)^2 
	:= \sum_{T \in \UU_H} \eta_H(T, v_H)^2 
	\quad \forall \UU_H \subseteq \TT_H
\end{align}
as well as the weighted-residual a posteriori error estimator
\begin{align}
	\eta_H := \eta_H(\TT_H)
	\quad \text{with} \quad 
	\eta_H(\UU_H) := \eta_H(\UU_H, u_H)
	\quad \forall \UU_H \subseteq \TT_H.
\end{align}
\end{subequations}

The following theorem provides an upper bound for the discretization error measured
in the energy norm
\begin{align}
\norm{v}{H^1_\kappa(\Omega)}^2 := a(v,v) = \norm{\kappa v}{\Omega}^2 + \norm{\nabla v}{\Omega}^2 
\quad \forall v\in H_0^1(\Omega)
\end{align}
in terms of the estimator $\eta_H$.

In the proof, more precisely in Step~2 and later in \emph{discrete} reliability
(Lemma~\ref{lem:discrete_reliability}), a key challenge associated with the (potential)
unboundedness of $\Omega$ arises. It turns out that elements touching the artificial boundary
$\Gamma_H$ are required be ``large''. This observation is in line with results
previously obtained for an equilibrated-flux estimator in~\cite{chaumontfrelet_2025b}.
As argued there, requiring these elements to be ``large'' is not a significant restriction,
and can be viewed as a generalized shape-regularity assumption.
In Section~\ref{sec:refinement} of the present manuscript, we will introduce a mesh
refinement strategy that guarantees this property.

\begin{theorem}[reliability]\label{thm:reliability}
There exists a constant $\const{rel}>0$ which depends only on the dimension $d$,
shape regularity of $\widehat\TT_H$, and $\min\set{\kTm h_T}{T\in\TT_H, T\cap \Gamma_H \neq \emptyset}$
such that 
\begin{align}\label{eq:weak_reliability}
\begin{split}
a(u - u_H,v)
\le
\const{rel}\left ( \eta_H \norm{v}{H^1_\kappa(\Omega)} + \dual{f}{v}_{\Omega\setminus\Omega_H}\right ) \quad \forall v\in H_0^1(\Omega).
\end{split}
\end{align}
In particular, coercivity of $a(\cdot,\cdot)$ yields reliability
\begin{align}\label{eq:reliability}
\begin{split}
\norm{u - u_H}{H^1_\kappa(\Omega)}
\le
\const{rel}\left (
\eta_H + \sup_{v \in H_0^1(\Omega)\setminus\{0\}} \frac{|\dual{f}{v}_{\Omega\setminus\Omega_H}|}{\norm{v}{H^1_\kappa(\Omega)}}
\right ).
\end{split}
\end{align}
% for $\const{rel} = \max\{\norm{\kappa^{-1}}{L^\infty(\Omega)},1\} \max\{\const{rel}',1\}$.
\end{theorem}

\begin{proof}
We prove the assertion in two steps.

\noindent
\textbf{Step~1:}
Galerkin orthogonality gives that 
\begin{align*}
	a(u - u_H,v) = a(u - u_H,v - v_H) \quad \forall v_H \in \SS_0^p(\TT_H).
\end{align*}
We split $a(u - u_H,v - v_H)$ as follows 
\begin{align*}
	&a(u - u_H,v - v_H) \\
	&\quad= \dual{f}{v}_{\Omega\setminus\Omega_H} + \dual{f}{v - v_H}_{\Omega_H} - \dual{\kappa^2 u_H}{v - v_H}_{\Omega_H} - \dual{\nabla u_H}{\nabla (v - v_H)}_{\Omega_H}.
\end{align*}
The first term $\dual{f}{v}_{\Omega\setminus\Omega_H}$ already appears in the reliability estimate~\eqref{eq:weak_reliability}. 
For the remaining terms, integration by parts shows that 
\begin{align*}
	&\dual{f}{v - v_H}_{\Omega_H} - \dual{\kappa^2 u_H}{v - v_H}_{\Omega_H} - \dual{\nabla u_H}{\nabla (v - v_H)}_{\Omega_H}
	\\
	&\quad = \sum_{T \in \TT_H} \dual{f}{v - v_H}_T - \dual{\kappa^2 u_H}{v - v_H}_T + \dual{\Delta u_H}{v - v_H}_T - \dual{\partial_{\bs{n}} u_H}{v - v_H}_{\partial T}
	\\
	&\quad = \sum_{T \in \TT_H} \dual{f - \kappa^2 u_H + \Delta u_H}{v - v_H}_T - \frac{1}{2} \dual{\jump{\partial_{\bs{n}} u_H}}{v - v_H}_{\partial T \cap \Omega_H} 
	\\
	&\qquad\qquad- \dual{\partial_{\bs{n}} u_H}{v - v_H}_{\partial T \cap \Gamma_H}. 
\end{align*}
In contrast to the standard case of bounded domains, the last term $\dual{\partial_{\bs{n}} u_H}{v - v_H}_{\partial T \cap \Gamma_H}$ does not necessarily vanish, as $v - v_H$ is in general only $0$ on $\partial\Omega$ but not on $\partial\Omega_H$. 
Note the identities $(\partial T\cap\Omega_H) \cup (\partial T \cap \Gamma_H) = \partial T \setminus \partial \Omega = \partial T \cap \Omega$ and  $\partial_{\bs{n}} u_H = \jump{\partial_{\bs{n}} u_H}$ on $\partial T\cap\Gamma_H$. 
We apply the Cauchy--Schwarz inequality to see that 
\begin{align*}
	&\sum_{T \in \TT_H} \dual{f - \kappa^2 u_H + \Delta u_H}{v - v_H}_T - \frac{1}{2} \dual{\jump{\partial_{\bs{n}} u_H}}{v - v_H}_{\partial T \cap \Omega_H} 
	\\
	&\qquad- \dual{\partial_{\bs{n}} u_H}{v - v_H}_{\partial T \cap \Gamma_H}
	\\
	&\quad \le \Big(\sum_{T \in \TT_H} h_T^2 \norm{f - \kappa^2 u_H + \Delta u_H}{T}^2 + h_T \norm{\jump{\partial_{\bs{n}} u_H}}{\partial T\cap \Omega}^2\Big)^{1/2}
	\\
	&\qquad \times \Big(\sum_{T\in\TT_H} h_T^{-2} \norm{v - v_H}{T}^2 + h_T^{-1} \norm{v - v_H}{\partial T \cap \Omega}^2 \Big)^{1/2}. 
\end{align*}

\noindent
\textbf{Step~2:}
We conclude the proof choosing $v_H := I_H v$ for some Cl\'ement-type interpolation operator $I_H: H_0^1(\Omega) \to \SS_0^1(\TT_H)$,
with the following local approximation property
\begin{align}\label{eq:local_approximation}
	h_T^{-2} \norm{(1 - I_H) w}T^2 + h_T^{-1} \norm{(1 - I_H) w}{\partial T\cap\Omega}^2
	\lesssim \norm{w}{H^1_\kappa(\widehat\omega_H[T])}^2
	\quad \forall w \in H_0^1(\Omega), T\in\TT_H,
\end{align}
where $\widehat\omega_H[T] := \bigcup \set{T'\in\widehat\TT_H}{T\cap T' \neq \emptyset}$ denotes the patch of $T$. 
A suitable candidate is found in the proof of \emph{discrete} reliability (Lemma~\ref{lem:discrete_reliability}) below. 
\end{proof}

\begin{remark}[reliability on full triangulation]\label{rem:extended_estimator}
Trivially, the supremum in \eqref{eq:reliability} can be estimated by
\begin{align*}
\sup_{v\in H_0^1(\Omega)\setminus\{0\}}
\frac{|\dual{f}{v}_{\Omega\setminus\Omega_H}|}{\norm{v}{H^1_\kappa(\Omega)}}
\le
\norm{\kappa^{-1}f}{\Omega\setminus\Omega_H}
=
\norm{\kappa^{-1}(f - \kappa^2 u_H + \Delta u_H)}{\Omega\setminus\Omega_H},
\end{align*}
where the last identity follows from the inclusion $\supp(u_H) \subseteq \overline\Omega_H$. 
Suppose there exist constants $\const{inf},\const{sup} > 0$
such that
\begin{align}\label{eq:infsup_h}
	\const{inf} \le \kTm h_T, \qquad \kTp h_T \le \const{sup} \quad \forall T\in\widehat\TT_H\setminus\TT_H.
\end{align}
Then, the definition of the indicators and the estimator in~\eqref{eq:estimator} can be readily extended to $T\in\widehat\TT_H$ and $\UU_H\subseteq\widehat\TT_H$, and the resulting estimator $\widehat\eta_H = \widehat\eta_H(\widehat\TT_H,u_H) < \infty$ is reliable, i.e., 
\begin{align*}
	\const{rel}^{-2} \norm{u - u_H}{H^1_\kappa(\Omega)}^2 
	&\le 2 \eta_H^2 + 2 \const{inf}^{-2} \sum_{T\in\widehat\TT_H\setminus\TT_H}  h_T^2\norm{f - \kappa^2 u_H + \Delta u_H}{T}^2 
	\\
	&\le 2 \widehat \eta_H(\TT_H)^2 + 2 \const{inf}^{-2} \widehat \eta_H(\widehat\TT_H\setminus \TT_H)^2
	\le 2 \max\{1,\const{inf}^{-2}\} \widehat \eta_H^2. 
\end{align*}
The upper bound in~\eqref{eq:infsup_h} ensures that $\widehat\eta_H < \infty$ in
case the support of $f$ is unbounded (note that $\sum_{T\in\widehat\TT_H}  h_T^2\norm{f}{T}^2$
might not be finite even if $f \in L^2(\Omega)$ in this case).
Together with uniform shape regularity~\eqref{eq:shape_regularity} of $\widehat\TT_H$,
the lower bound in \eqref{eq:infsup_h} also guarantees that
$\min\set{\kTm h_T}{T\in\TT_H, T\cap \Gamma_H \neq \emptyset}$ in
Theorem~\ref{thm:reliability} is uniformly bounded from below. 
\end{remark}

We next show that the estimator is also a \emph{local} lower bound for the error plus oscillations.

\begin{theorem}[efficiency]\label{thm:efficiency}
Let $q\ge \max\{p-2,0\}$ and $\Pi_H: L^2(\Omega) \to \PP^q(\widehat \TT_H)$ be the $L^2$-projection onto the space of piecewise polynomials of degree $q$. 
Then, there exists a constant $\const{eff}>0$ which depends only on the dimension $d$, the polynomial degrees $p$ and $q$, and shape regularity of $\widehat \TT_H$
such that 
\begin{align}\label{eq:efficiency}
\begin{split}
	\const{eff}^{-1}\eta_H(T,v_H)
	&\le \max_{T'\in\widehat \TT_H^*[F]} (\kTTp h_{T'}) \norm{\kappa(u - v_H)}{\widehat\omega_H^*[T]}  + \norm{\nabla (u - v_H)}{\widehat\omega_H^*[T]} 
	\\ 
	&\qquad + h_T \norm{(1-\Pi_H)(f-\kappa^2 v_H)}{\widehat\omega_H^*[T]} 
	 \quad \forall v_H \in \SS_0^p(\TT_H), T\in\widehat\TT_H,
\end{split}
\end{align}
where $\widehat\omega_H^*[T]:=\set{T'\in\widehat \TT_H}{T \text{ and }T' \text{ share a face}}$. 
\end{theorem}

\begin{proof}
The proof follows as for bounded domains $\Omega$ by the use of bubble functions and the weak formulation. 
It is split into three steps. 

\noindent
\textbf{Step~1:}
Let $T\in\widehat \TT_H$ and abbreviate the volume residual $r_T := (f - \kappa^2 v_H + \Delta v_H)|_T$. 
The triangle inequality shows that 
\begin{align*}
	h_T \norm{r_T}{T}
	\le h_T \norm{\Pi_H r_T}{T} + h_T \norm{(1-\Pi_H) r_T}{T}.
\end{align*}
Let $\beta_T := \prod_{z\in \VV_T} \phi_{H,z} \in \PP^3(T) \cap H_0^1(T)$ be the canonical bubble function on $T$, where $\VV_T$ denotes the set of vertices in $T$ and $\phi_{H,z}$ are the corresponding hat functions. 
A scaling argument, the definition of the weak solution $u$, and integration by parts for $\Delta v_H$ give that
\begin{align*}
	&\norm{\Pi_H r_T}{T}^2 
	\simeq \norm{\beta_T^{1/2} \Pi_H r_T}{T}^2 
	= \dual{r_T}{\beta_T \Pi_H r_T}_T - \dual{(1-\Pi_H) r_T}{\beta_T\Pi_H r_T}_T
	\\
	&= \dual{\kappa^2 (u-v_H)}{\beta_T \Pi_H r_T}_T + \dual{\nabla (u-v_H)}{\nabla (\beta_T \Pi_H r_T)}_T 
	\\
	&\quad - \dual{(1-\Pi_H) r_T}{\beta_T\Pi_H r_T}_T
	\\
	&\le \norm{\kappa^2 (u-v_H)}{T} \norm{\beta_T\Pi_H r_T}{T} + \norm{\nabla (u-v_H)}{T} \norm{\nabla (\beta_T \Pi_H r_T)}{T} 
	\\
	&\quad + \norm{(1-\Pi_H) r_T}{T} \norm{\beta_T\Pi_H r_T}{T}.
\end{align*}
With the inverse inequality $\norm{\nabla (\beta_T \Pi_H r_T)}{T} \lesssim h_T^{-1} \norm{\beta_T \Pi_H r_T}{T}$ and $0\le \beta_T \le 1$, we see that
\begin{align*}
	h_T \norm{r_T}{T}
	\lesssim \kTp h_T \norm{\kappa(u - v_H)}{T} + \norm{\nabla (u - v_H)}{T} + h_T \norm{(1-\Pi_H) r_T}{T}.
\end{align*}

\noindent
\textbf{Step~2:}
Let $T\in\widehat\TT_H$ and $F\subset \partial T \cap \Omega$ a face of it (if such a face exists).
Abbreviate the jump $j_F := \jump{\partial_{\bs{n}} v_H}|_F$.
Let $\beta_F := \prod_{z\in \VV_F} \phi_{H,z} \in \PP^2(\widehat \TT_H^*[F]) \cap H_0^1(\widehat\omega_H^*[F])$ be the canonical bubble function on the face patch $\widehat\omega_H^*[F]:=\bigcup \widehat \TT_H^*[F]$ with $\widehat \TT_H^*[F]:=\set{T\in\widehat \TT_H}{F \text{ is face of } T}$, where $\VV_F$ denotes the set of vertices in $F$ and $\phi_{H,z}$ are the corresponding hat functions. 
Moreover, let $\widetilde j_F\in\PP^{p-1}(\widehat \TT_H^*[F])$ be a stable piecewise polynomial extension in the sense that
\begin{align}\label{eq:trace_extension}
	{\widetilde j_F}|_F = j_F \quad \text{and} \quad \norm{\widetilde j_F}{\widehat\omega_H^*[F]} \lesssim h_T^{1/2} \norm{j_F}{F}.
\end{align}
The existence of such an extension follows from a simple scaling argument. 
Another scaling argument, the definition of the weak solution $u$, and integration by parts for $\partial_{\bs{n}} v_H$ give that
\begin{align*}
	&\norm{j_F}{F}^2 \simeq \norm{\beta_F^{1/2} j_F}{F}^2 
	= \dual{j_F}{\beta_F j_F}_F 
	= \Big| \sum_{T'\in\widehat \TT_H^*[F]} \dual{\partial_{\bs{n}} v_H}{\beta_F j_F}_{\partial T'} \Big|
	\\
	&= \Big| \sum_{T'\in\widehat \TT_H^*[F]} \dual{\Delta v_H}{\beta_F \widetilde j_F}_{T'} + \dual{\nabla v_H}{\nabla(\beta_F \widetilde j_F)}_{T'} \Big|
	\\
	&= \Big| \sum_{T'\in\widehat \TT_H^*[F]} \dual{f - \kappa^2 v_H + \Delta v_H}{\beta_F \widetilde j_F}_{T'} - \dual{\kappa^2 (u - v_H)}{\beta_F \widetilde j_F}_{T'} 
	\\
	&\qquad\qquad - \dual{\nabla (u - v_H)}{\nabla(\beta_F \widetilde j_F)}_{T'} \Big|
	\\
	&\le \sum_{T'\in\widehat \TT_H^*[F]} \norm{r_{T'}}{T'} \norm{\beta_F \widetilde j_F}{T'} + \norm{\kappa^2 (u - v_H)}{T'} \norm{\beta_F \widetilde j_F}{T'} 
	\\
	&\qquad\qquad+ \norm{\nabla (u - v_H)}{T'} \norm{\nabla(\beta_F \widetilde j_F)}{T'}.
\end{align*}
With Step~1 (for $T'$ instead of $T$), the inverse inequality $\norm{\nabla (\beta_F \widetilde j_F)}{T'} \lesssim h_{T'}^{-1} \norm{\beta_F \widetilde j_F}{T'}$, $0\le \beta_F \le 1$, and the stability~\eqref{eq:trace_extension}, we see that
\begin{align*}
	h_T^{1/2} \norm{j_F}{F}
	\lesssim \sum_{T'\in\widehat \TT_H^*[F]} h_{T'} \norm{r_{T'}}{T'} + (\kTTp h_{T'})\norm{\kappa(u - v_H)}{T'} + \norm{\nabla (u - v_H)}{T'}.
\end{align*}

\noindent
\textbf{Step~3:}
Combining Steps~1--2 and noting that $(1-\Pi_H)(\Delta v_H)|_T=0$ as well as $\widehat\omega_H^*[F] \subseteq \widehat\omega_H^*[T]$ for all faces $F$ of $T$ concludes the proof. 
\end{proof}

\begin{remark}[reliability on full triangulation]
With the same proof, the result of Theorem~\ref{thm:efficiency} holds analogously for all $T\in\widehat\TT_H$ and the extended estimator $\widehat\eta_H$ from Remark~\ref{thm:efficiency}. 
\end{remark}

%%%%%%%%%%%%%%%%%%%%%%%%%%%%%%%%%%%%%%%%%%%%%%%%%%%%%%%%%%%%%%%
\section{Refinement of infinite meshes}\label{sec:refinement}
%%%%%%%%%%%%%%%%%%%%%%%%%%%%%%%%%%%%%%%%%%%%%%%%%%%%%%%%%%%%%%%

In order to formulate an adaptive algorithm steered by the a posteriori error estimator of Section~\ref{sec:aposteriori}, we first need to 
introduce a suitable refinement strategy starting from some initial conforming triangulation $\widehat\TT_0$ of $\Omega$ and a corresponding finite subtriangulation $\TT_0$. 
We assume that $\widehat\TT_0$ is uniformly shape regular in the sense of \eqref{eq:shape_regularity}, and that the element sizes in $\widehat\TT_0$ are uniformly bounded from above and below in the sense of \eqref{eq:infsup_h}. 
In particular, this guarantees uniform reliability \eqref{eq:reliability} and efficiency \eqref{eq:efficiency} for the corresponding Galerkin approximation $u_0 \in \SS_0^p(\TT_0)$; see also Remark~\ref{rem:extended_estimator}.
From this point of view, it is essential that also the refined (infinite) triangulations still satisfy both \eqref{eq:shape_regularity} and \eqref{eq:infsup_h}. 

To this end, we can essentially use the refinement strategy of~\cite{maubach95,traxler97}, known as newest vertex bisection for $d=2$. 
We first recall the concept of tagged simplices, which are ordered tupels 
\begin{align*}
	T = [z_0, \dots, z_d; \tau],
\end{align*}
where $z_0, \dots, z_d \in \R^d$ do not lie on a $(d-1)$-dimensional hyperplane and $\tau \in \{0,\dots,d-1\}$, 
and define the refinement edge of $T$ as $E_T:={\rm conv}\{z_0,z_d\}$. 
To refine $T$, which we identify with the simplex ${\rm conv}\{z_0,\dots,z_d\}$, it is bisected along $E_T$ into the two children
\begin{align*}
  C_1(T) &:= [z_0, \mbox{$\frac{z_0+z_d}{2}$}, z_1, \dots, z_\tau, z_{\tau+1}, \dots, z_{d-1}; \, {\rm mod}(\tau+1, d)], \\
  C_2(T) &:= [z_d, \mbox{$\frac{z_0+z_d}{2}$}, z_1, \dots, z_\tau, z_{d-1}, \dots, z_{\tau+1}; \, {\rm mod}(\tau+1, d)],
\end{align*}
increasing the type $\tau$ by one. 
From now on, we suppose that all elements in $\widehat\TT_0$ are tagged and have type $\tau = 0$. 
For any successor $T$ of an element $T_0 \in \widehat\TT_0$, we define $\level(T) \in \N_0$ as the number of bisections needed to obtain $T$ from $T_0$ and $\level(T) = 0$ if and only if $T = T_0 \in \TT_0$.
Since the considered bisection rule only yields a finite number of different shapes~\cite{maubach95,traxler97}, there particularly holds the shape regularity  
\begin{align}\label{eq:level2diam}
	2^{-\level(T)} |T_0| \le |T| \le \diam(T)^d \le \const{sha} |T| = 2^{-\level(T)} \const{sha} |T_0|
\end{align}
for a uniform constant $\const{sha}>0$ that depends only on the shape-regularity constant of $\widehat \TT_0$.

Let $\widehat\T$ denote the set of all conforming triangulations $\widehat\TT_H$ that are reachable from $\widehat\TT_0$ by iterative (but finite) use of the above bisection rule. 
Similarly, we denote by $\widehat\T(\widehat\TT_H)$ all conforming triangulations $\widehat\TT_h$ reachable from $\widehat\TT_H$. 
Under an admissibility assumption on the initial mesh~\cite[Section~4]{stevenson08}, \cite{stevenson08} considers a recursive refinement algorithm for $\widehat\TT_H\in\widehat\T$ and finite $\MM_H\subset \widehat\TT_H$ which terminates, using only \emph{finitely} many bisections, and generates the coarsest conforming refinement $\refine(\widehat\TT_H,\MM_H)\in\widehat\T(\widehat\TT_H)$ of $\widehat\TT_H$ such that all $T\in\MM_H$ are bisected at least once. 
We stress that \cite{stevenson08} indeed allows for general open (\emph{not} necessarily bounded) sets $\Omega$ and \emph{locally} finite triangulations thereof; see~\cite[Section~3]{stevenson08}. 
In particular, \cite{stevenson08} provides the following mesh-closure estimate, which is an essential ingredient for the optimality proof later; see, e.g., \cite{cfpp14}. 
For $d=2$, the estimate is due to \cite{bdd04}. 

\begin{proposition}[mesh-closure estimate{~\cite[Theorem~6.1]{stevenson08}}]\label{prop:closure}
% \cite[Theorem~2.4]{bdd04}, \cite[Theorem~6.1]{stevenson08}, \cite[Theorem~2]{kpp13}
Let $\widehat \TT_0$ be admissible with finite shape-regularity constant and uniformly
bounded element sizes, i.e.,
$0<\inf_{T\in\widehat \TT_0} h_{T}$ and $\sup_{T\in\widehat \TT_0} h_{T} < \infty$. 
Then, there exists $\const{clos}>0$ such that for any sequence of successively refined
triangulations $\widehat \TT_{\ell+1} := \refine(\widehat \TT_\ell, \MM_\ell)$ with
arbitrary finite $\MM_\ell \subseteq \widehat \TT_\ell$ for all $\ell \in \N_0$, it holds that
\begin{align}\label{eq:mesh_closure}
	\# (\widehat \TT_\ell \setminus \widehat \TT_0)
	\le \const{clos} \sum_{j=0}^{\ell-1} \# \MM_j
	\quad \forall \ell \in \N.
\end{align}
The constant $\const{clos}$ depends only on the dimension $d$, the shape-regularity constant of $\widehat\TT_0$, and the ratio $\sup_{T\in\widehat \TT_0} h_{T} / \inf_{T\in\widehat \TT_0} h_{T}$. 
\qed
\end{proposition}

While the following overlay estimate has seemingly only been formulated for bounded $\Omega$ in the literature~\cite{stevenson07,ckns08}, the standard proof extends in a straight-forward way.

\begin{proposition}[overlay estimate]\label{prop:overlay}
Let $\widehat \TT_0$ be admissible with finite shape-regularity constant and finite mesh size. 
Then, for all $\widehat\TT_H, \widehat\TT_h \in \widehat\T$, there exists a common refinement $\widehat\TT_H \oplus \widehat\TT_h \in \widehat\T(\widehat\TT_H) \cap \widehat\T(\widehat\TT_h)$ such that
\begin{align}\label{eq:overlay}
	\# \big((\widehat \TT_H \oplus \widehat \TT_h) \setminus \widehat \TT_h\big) \le 2\# (\widehat \TT_H \setminus \widehat \TT_0).
\end{align}
\end{proposition}

\begin{proof}
As already mentioned, the proposition follows as in the standard case of bounded $\Omega$; see~\cite[proof of Lemma~5.2]{stevenson07} or \cite[Lemma~3.7]{ckns08}. 
The overlay is given by
\begin{align*}
	\widehat \TT_H \oplus \widehat \TT_h := \set{T\in\TT_H}{\exists T' \in \TT_h \quad T \subseteq T'}  \cup \set{T\in\TT_h}{\exists T' \in \TT_H \quad T \subseteq T'}.
\end{align*}
We also give an alternative proof, which directly employs the standard result: 
Define the refined region
\begin{align*}
	\Omega_{H,h} 
	:= \bigcup (\widehat \TT_H \setminus \widehat \TT_0) \cup \bigcup (\widehat \TT_h \setminus \widehat \TT_0)
	= \bigcup (\widehat \TT_0 \setminus \widehat \TT_H) \cup \bigcup (\widehat \TT_0 \setminus \widehat \TT_H).
\end{align*}
Note that $\# (\widehat \TT_H \setminus \widehat \TT_0) < \infty$ and thus $\# (\widehat \TT_0 \setminus \widehat \TT_H) < \infty$. 
The same is true for $\widehat \TT_h$ instead of $\widehat \TT_H$.
We can employ the standard overlay estimate for the bounded set $\Omega_{H,h}$ with finite initial triangulation $(\widehat \TT_0 \setminus \widehat \TT_H) \cup (\widehat \TT_0 \setminus \widehat \TT_H)$, i.e., 
\begin{align*}
	\# \big((\widehat \TT_H|_{\Omega_{H,h}} \oplus \widehat \TT_h|_{\Omega_{H,h}}) \setminus \widehat \TT_h|_{\Omega_{H,h}}\big) \le 2\# (\widehat \TT_H|_{\Omega_{H,h}} \setminus \widehat \TT_0|_{\Omega_{H,h}}),
\end{align*}
where we use the notation $\widehat \TT|_{\Omega_{H,h}} := \set{T\in\widehat\TT}{T\subseteq\overline\Omega_{H,h}}$ for any mesh $\widehat\TT$. 
To be precise, the overlay estimate is usually stated as 
\begin{align*}
	\# (\widehat \TT_H|_{\Omega_{H,h}} \oplus \widehat \TT_h|_{\Omega_{H,h}}) \le \# \widehat \TT_H|_{\Omega_{H,h}} + \# \widehat \TT_h|_{\Omega_{H,h}} - \# \widehat \TT_0|_{\Omega_{H,h}},
\end{align*}
but the fact that refined elements are split into at least two children readily implies that 
\begin{align*}
	\# \big((\widehat \TT_H|_{\Omega_{H,h}} \oplus \widehat \TT_h|_{\Omega_{H,h}}) \setminus \widehat \TT_h|_{\Omega_{H,h}}\big) 
	&\le 2 \big(\#(\widehat \TT_H|_{\Omega_{H,h}} \oplus \widehat \TT_h|_{\Omega_{H,h}})  - \#\widehat \TT_h|_{\Omega_{H,h}}\big) 
	\\
	&\le 2 (\# \widehat \TT_H|_{\Omega_{H,h}} - \# \widehat \TT_0|_{\Omega_{H,h}}) 
	\\
	&\le 2\# (\widehat \TT_H|_{\Omega_{H,h}} \setminus \widehat \TT_0|_{\Omega_{H,h}}).
\end{align*}
Clearly, we have that $\widehat \TT_H|_{\Omega_{H,h}} \oplus \widehat \TT_h|_{\Omega_{H,h}} = (\widehat \TT_H\oplus \widehat \TT_h)|_{\Omega_{H,h}}$. 
Hence, the identities $(\widehat \TT_H\oplus \widehat \TT_h)|_{\Omega_{H,h}} \setminus \widehat \TT_h|_{\Omega_{H,h}} = (\widehat \TT_H \oplus \widehat \TT_h) \setminus \widehat \TT_h$ 
and $\widehat \TT_H|_{\Omega_{H,h}} \setminus \widehat \TT_0|_{\Omega_{H,h}}=\widehat \TT_H \setminus \widehat \TT_0$ yield the desired overlay estimate~\eqref{eq:overlay}. 
Finally, the inclusion $\widehat \TT_H \oplus \widehat \TT_h \in \widehat\T(\widehat\TT_H) \cap \widehat\T(\widehat\TT_h)$ is equivalent to $\widehat \TT_H \oplus \widehat \TT_h$ being conforming and reachable from $\widehat\TT_H$ and $\widehat\TT_h$ by elementwise bisection.
Conformity and reachability follow from the facts that $(\widehat \TT_H\oplus \widehat \TT_h)|_{\Omega_{H,h}}$ is conforming and reachable from $\widehat\TT_H$ and $\widehat\TT_h$  by elementwise bisection 
and that no refinements outside of $\Omega_{H,h}$ take place to reach the conforming triangulations $\widehat\TT_H$ and $\widehat\TT_h$ from $\widehat\TT_0$. 
\end{proof}

\begin{remark}[alternative proof of closure estimate]
Building only on the result for finite triangulations, an extension of the closure estimate from Proposition~\ref{prop:closure} to our setting can alternatively be seen similarly as in the proof of the overlay estimate.
\end{remark}

\begin{remark}[admissibility condition]
While for $d>2$, it is not clear that the simplices in $\widehat\TT_0$ can be tagged in such a way that $\widehat\TT_0$ becomes admissible, admissibility can be ensured if the centroid of each element is added; see~\cite[Appendix~A]{stevenson07}.
In the recent work~\cite{dgs25}, a simple initialization of the Maubach bisection routine~\cite{maubach95} is proposed which applies to any conforming initial triangulation. 
It is shown that Maubach’s routine with this initialization always terminates and satisfies the closure estimate. 
While \cite{dgs25} only consider \emph{finite} triangulations, a generalization to the present setting might be feasible.
In this case, also that refinement strategy can be applied in Section~\ref{sec:adaptivity}. 
\end{remark}

%%%%%%%%%%%%%%%%%%%%%%%%%%%%%%%%%%%%%%%%%%%%%%%%%%%%%%%%%%%%%%%
\section{Adaptive algorithm and optimal convergence}
\label{sec:adaptivity}
%%%%%%%%%%%%%%%%%%%%%%%%%%%%%%%%%%%%%%%%%%%%%%%%%%%%%%%%%%%%%%%

We consider the following adaptive algorithm, where the employed refinement strategy of Section~\ref{sec:refinement} guarantees that it is practical in the sense that the generated subtriangulations $\TT_\ell$ are finite. 

\begin{algo}[Adaptive mesh refinement]
\label{alg:adaptive}
\\
\textbf{Input:} Initial admissible triangulation $\widehat \TT_0$ of $\Omega$ with finite shape-regularity constant and uniformly bounded element sizes;  
finite subset of active elements $\TT_0 \subset \widehat \TT_0$; right-hand side $f$ with $\supp f \subseteq \bigcup \TT_0$; polynomial degree $p$; D\"orfler parameter $0<\theta\le1$.\\
\textbf{Loop:} Iterate the following four steps for all $\ell=0,1,2,\dots$:
\begin{enumerate}[\rm(i)]
\item Compute Galerkin approximation $u_\ell\in\SS_0^p(\TT_\ell)$ of $u\in H_0^1(\Omega)$. 
\item Compute error indicators $\eta_\ell(T)$ for all $T\in\TT_\ell$. 
\item Determine a minimal set of marked elements such that $\theta \eta_\ell^2 \le \eta_\ell(\MM_\ell)^2$.
\item Generate refinement $\widehat \TT_{\ell+1} := \refine(\widehat \TT_\ell, \MM_\ell)$ and activate all refined elements, i.e., $\TT_{\ell+1} := \widehat \TT_{\ell+1} \setminus (\widehat \TT_\ell \setminus \TT_\ell) = \widehat \TT_{\ell+1} \setminus (\widehat \TT_0 \setminus \TT_0)$. 
\end{enumerate}
\textbf{Output:} Triangulations $(\widehat\TT_\ell)_{\ell\in\N_0}$; subtriangulations $(\TT_\ell)_{\ell\in\N_0}$; Galerkin approximations $(u_\ell)_{\ell\in\N_0}$;  error estimators $(\eta_\ell)_{\ell\in\N_0}$. 
\end{algo}

%\gg{What to do with the boundary? 
%\begin{itemize}
%\item The current algorithm just refines boundary elements as usual and only pushes if hanging nodes are created on the boundary. That is the minimal requirement to preserve extended shape regularity and thus reliability. 
%\item We could only push the boundary to refine marked boundary elements and not refine them. What to do if a boundary element is not marked but must be refined to avoid hanging nodes?
%\item We could always push and refine boundary elements if they are marked or should be refined to avoid hanging nodes. 
%\item We could compare the volume and jump term to decide what to do.
%\end{itemize}}

\begin{remark}[comparison with the algorithm from~\cite{chaumontfrelet_2025b}]
In~\cite{chaumontfrelet_2025b}, one of the authors previously introduced an
adaptive algorithm for unbounded domains based on a flux-equilibrated estimator.
In that algorithm, a truncation box $(-n,n)^d$ is explicitly used to
define $\Omega_H := \Omega \cap (-n,n)^d$, and the parameter $n$ is
increased in the mesh-refinement procedure when elements close to the artificial
boundary are marked. A drawback of this approach is that the artificial boundary
is isotropically pushed towards infinity, which may not be optimal. In addition,
it is not clear that optimality convergence rates can be shown (nor that they hold true)
for the algorithm in~\cite{chaumontfrelet_2025b}. Algorithm~\ref{alg:adaptive}
is therefore a significant improvement since the artificial boundary is locally pushed towards infinity, which in turn allows for optimal convergence rates that we rigorously establish below.
\end{remark}

Building on the axioms of adaptivity from~\cite{cfpp14} for the error estimator $\eta$, which we will verify in Section~\ref{sec:axioms}, we will show R-linear as well as optimal convergence of Algorithm~\ref{alg:adaptive}. 
The proofs are given in Section~\ref{sec:proofs}. 
Since the abstract framework of~\cite{cfpp14} is essentially designed for finite triangulations of a fixed bounded Lipschitz domain, both the axioms and the corresponding convergence proofs must be adapted to the present setting. 

\begin{theorem}[R-linear convergence]\label{thm:linear_convergence}
There exist $\const{lin}>0$ and $0<q_{\rm red}<1$ such that 
\begin{align}\label{eq:linear_convergence}
	\eta_{\ell+n} \le \const{lin} q_{\rm lin}^n \eta_\ell \quad \forall \ell,n\in\N_0.
\end{align}
The constants $\const{lin}$ and $q_{\rm lin}$
depend only on the dimension $d$, the polynomial degree $p$,
the shape regularity of $\widehat\TT_0$, the constants
$\sup_{T\in \widehat \TT_0} (\kTp h_T)$ and
% $\min\set{\kTm h_T}{T\in\TT_0, T \cap \Gamma_0 \neq \emptyset}$
$\inf\set{\kappa_{\widehat\omega_0[T]}^- h_T}{T\in \widehat{\TT}_0 \setminus \TT_0}$,
and the D\"orfler parameter $\theta$. Here, $\kappa_{\widehat\omega_0[T]}^-$ is the essential
infimum of $\kappa$ on the patch
$\widehat\omega_0[T] := \bigcup \set{T'\in\widehat\TT_0}{T\cap T' \neq \emptyset}$.
\end{theorem}

To formulate optimal convergence, we first need to introduce suitable approximation classes. 
For any $\widehat\TT_H \in \widehat\T$, we abbreviate the finite subtriangulation $\TT_H := \widehat \TT_H \setminus (\widehat \TT_0 \setminus \TT_0)$, which is consistent with the notation from Algorithm~\ref{alg:adaptive}, and the set of all reachable subtriangulations $\T:=\set{\TT_H}{\widehat \TT_H \in \widehat \T}$. 
For $s>0$, we define
\begin{align}
	\norm{u}{\A_s} := \sup_{\epsilon>0} \min_{\TT_H \in \T_\epsilon(u)} \#(\TT_H \setminus \TT_0)^s \epsilon
	 \quad \text{with} \quad \T_\epsilon(u) := \set{\TT_H \in \T}{\eta_H \le \epsilon}. 
\end{align}
Note that $\norm{u}{\A_s}$ can be equivalently defined via the extended meshes $\widehat\TT_H$ since $\TT_H \setminus \TT_0 = \widehat\TT_H \setminus \widehat\TT_0$. 
Moreover, we mention the equivalence 
\begin{align}
	\# \TT_H - \#\TT_0 \le \#(\TT_H \setminus \TT_0) \le 2(\# \TT_H - \#\TT_0),
\end{align}
where the last inequality follows from the fact that refined elements are at least split once. 
We also mention that the approximation class can be equivalently defined by interchanging the roles of the number of elements and the estimator; see, e.g., \cite[Lemma~5.1]{ffgp19}.

\begin{theorem}[optimal convergence]\label{thm:optimal_convergence}
Let $0<\theta<\theta_{\rm opt}:=\frac{1}{1+\const{stab}^2\const{drel}^2}$ (with the constants $\const{stab}>0$ and $\const{drel}>0$ from  Lemma~\ref{lem:stability} and~\ref{lem:discrete_reliability} below) and $s>0$.
Then, there exists $\const{opt}>0$ such that 
\begin{align*}
	\#(\TT_\ell \setminus \TT_0)^s \eta_\ell 
	\le \const{opt} \norm{u}{\A_s}
	\quad \forall \ell\in\N_0.
\end{align*}
The constant $\const{opt}$ depends only on the dimension $d$, the polynomial degree $p$,
shape regularity of $\widehat\TT_0$, the constants
$\sup_{T\in \widehat \TT_0} (\kTp h_T)$, 
$\inf\set{\kappa_{\widehat\omega_0[T]}^- h_T}{T \in \widehat{\TT}_0\setminus \TT_0}$,
and $\sup_{T\in\widehat \TT_0} h_{T} / \inf_{T\in\widehat \TT_0} h_{T}$
and the D\"orfler parameter $\theta$, and the rate $s$. 
\end{theorem}

While we only activate elements $\widehat\TT_\ell \setminus \TT_\ell$ if they are bisected to ensure conformity of the next mesh $\widehat\TT_{\ell+1}$, depending on the exact solution $u$, it might be more beneficial to also allow for activation of elements that are not refined. 
Given $\widehat\TT_H \in \widehat\T$, an \emph{arbitrary} finite subtriangulation $\TT_H\subset \widehat\TT_H$ with $\widehat\TT_H \setminus \TT_H \subset \widehat\TT_0$, and marked elements $\MM_H \subseteq \TT_H$, we define
\begin{align}
	&\push(\TT_H, \MM_H) := \TT_H \cup \set{T\in\widehat\TT_H}{\exists M \in \MM_H \text{ s.t. } T \text{ and }M \text{ share a face}},
	\\
	&\refine(\TT_H, \MM_H) := \refine(\widehat\TT_H,\MM_H) \setminus (\widehat\TT_H \setminus \TT_H).
\end{align}
By $\T^{\rm pr}=\T^{\rm pr}(\TT_0)$, we denote the set of all finite triangulations $\TT_H^{\rm pr}$ reachable via $\push$ and $\refine$ in the sense that there exist $(\widehat\TT_j^{\rm pr})_{j=1}^n$ and finite $\TT_j^{\rm pr}\subset \widehat\TT_j^{\rm pr}$ such that $\widehat \TT_n^{\rm pr} = \widehat \TT_H^{\rm pr}$ and $\widehat \TT_{j+1}^{\rm pr} = \refine(\push(\TT_j^{\rm pr},\MM_j^{\rm push}),\MM_j^{\rm ref})$ for $j=0,\dots,n-1$ and some $\MM_j^{\rm push},\MM_j^{\rm ref}\subseteq\TT_j^{\rm pr}$. 
Here, $\TT_0^{\rm pr}:=\TT_0$ is again the fixed initial set of active elements from Algorithm~\ref{alg:adaptive}. 
The next proposition states that, for all $s>0$, $\norm{u}{\A_s}$ is equivalent to 
\begin{align}
	\norm{u}{\A_s^{\rm pr}} := \sup_{\epsilon>0} \min_{\TT_H^{\rm pr} \in \T_\epsilon^{\rm pr}(u)} \#(\TT_H^{\rm pr} \setminus \TT_0)^s \epsilon
	\,\,\, \text{with} \,\,\, \T_\epsilon^{\rm pr}(u) := \set{\TT_H^{\rm pr} \in \T^{\rm pr}}{\eta_H^{\rm pr} \le \epsilon}. 
\end{align}
In particular, this shows that any rate that can be achieved by a combination of pushing and refining can also be achieved by refining---and implicitly pushing, i.e., activating refined elements---only. 
The proof only hinges on the mesh-closure estimate of Proposition~\ref{prop:closure} and quasi-monotonicity of the estimator, i.e.,
\begin{align}\label{eq:quasi-monotonicity}
	\eta_h^{\rm pr} \le \const{mon} \eta_H^{\rm pr} \quad \forall \TT_H^{\rm pr} \in\T^{\rm pr}, \TT_h^{\rm pr}\in\T^{\rm pr}(\TT_H)
\end{align}
for some uniform constant $\const{mon}>0$, which is a simple consequence of the axioms stability~(A1), reduction~(A2), and discrete reliability~(A3), shown in the following Section~\ref{sec:axioms}; see, e.g., \cite[Lemma 3.5]{cfpp14}. 
Here, $\T^{\rm pr}(\TT_H^{\rm pr})$ is defined analogously as $\T^{\rm pr}(\TT_0)$. 

\begin{proposition}[equivalence of approximation classes]
For all $s>0$, there exists a constant $\const{push}>0$ depending only on the dimension $d$, the mesh-closure constant $\const{clos}$, the quasi-monotonicity constant $\const{mon}$,  and the parameter $s$ such that
\begin{align}\label{eq:approximation_push}
	\norm{u}{\A_s^{\rm pr}} \le \norm{u}{\A_s} \le \const{push} \norm{u}{\A_s^{\rm pr}}.
\end{align}
\end{proposition}

\begin{proof}
Clearly, it holds that $\T \subset \T^{\rm pr}$ so that the first inequality in~\eqref{eq:approximation_push} is trivially satisfied. 
We split the proof of the second inequality into two steps. 

\noindent
\textbf{Step~1:}
We show the existence of a uniform constant $N\in\N$ such that for all $\TT_H^{\rm pr} \in \T^{\rm pr}$, there exists $\TT_H \in \T \cap \T^{\rm pr}(\TT_H^{\rm pr})$ with
\begin{align} 
	%\#(\TT_H^{\rm pr} \setminus \TT_0) \le 
	\#(\TT_H \setminus \TT_0) \le N \const{clos} \#(\TT_H^{\rm pr} \setminus \TT_0).
\end{align}
Using the notation from the definition of $\T^{\rm pr}$, we set $\MM_j := \MM_j^{\rm push} \cup \MM_j^{\rm ref}$. 
By potentially reducing $\MM_j^{\rm push}$, we may assume that $\# \MM_j^{\rm push} \le \#(\push(\TT_j^{\rm pr},\MM_j^{\rm push})\setminus \TT_j^{\rm pr})$, i.e., each marked element yields the activation of at least one new element. 
Moreover, there exists a uniform constant $N\in\N$ depending only on $d$ such that $N$ bisections of any element $T$ create an interior vertex in all faces of $T$. 
We denote the corresponding refinement strategy (plus conforming closure) on $\widehat\T$ by $\refine^N$ and set $\widehat\TT_{j+1} := \refine^N(\widehat\TT_j,\MM_j\cap \widehat\TT_j)$ as well as $\TT_{j+1} := \widehat\TT_{j+1} \setminus (\widehat \TT_0 \setminus \TT_0)$.
Since $\refine^N$ implicitly ``pushes'' all marked elements, it is clear that $\TT_H:=\TT_n$ is finer than $\TT_H^{\rm pr}$, i.e., $\TT_H \in \T^{\rm pr}(\TT_H^{\rm pr})$. 
Furthermore, the closure estimate (Proposition~\ref{prop:closure}) gives that 
\begin{align*}
	\#(\TT_H \setminus \TT_0) 
	= \#(\widehat\TT_H \setminus \widehat\TT_0) 
	\le N \const{clos} \sum_{j=0}^{n-1} \#\MM_j. 
\end{align*}
The assumption on $\MM_j^{\rm push}$ and the fact that elements in $\MM_j^{\rm ref}$ are bisected at least once to obtain $\#\TT_{j+1}^{\rm pr}$ yield that 
\begin{align*}
	\sum_{j=0}^{n-1} \#\MM_j
	&\le \sum_{j=0}^{n-1} (\#\MM_j^{\rm push} + \#\MM_j^{\rm ref})
	\\
	&\le \sum_{j=0}^{n-1} (\#\TT_{j+1}^{\rm pr} - \#\TT_j^{\rm pr})
	= \#\TT_H^{\rm pr} - \#\TT_0
	\le \#(\TT_H^{\rm pr} \setminus \TT_0). 
\end{align*}

\noindent
\textbf{Step~2:}
Let $\epsilon>0$ and let $\TT_H^{\rm pr} \in \T^{\rm pr}_\epsilon(u)$ be the minimizer of $\#(\TT_H^{\rm pr} \setminus \TT_0)$. 
Then quasi-monotonicity~\eqref{eq:quasi-monotonicity} implies for the mesh $\TT_H$ of Step~1 that $\eta_H \le \const{mon}\eta_H^{\rm pr} \le \const{mon}\epsilon$, i.e., $\TT_H\in\T_{\const{mon}\epsilon}(u)$.
With Step~1, we conclude that 
\begin{align*}
	\norm{u}{\A_s} 
	&= \sup_{\epsilon>0} \min_{\TT_H\in\T_{\const{mon}\epsilon}(u)} \#(\TT_H \setminus \TT_0)^s \const{mon}\epsilon 
	\\
	&\le \const{mon} \sup_{\epsilon>0} \min_{\TT_H^{\rm pr} \in \T_\epsilon^{\rm pr}(u)}(N \const{clos})^s \#(\TT_H^{\rm pr} \setminus \TT_0)^s \epsilon
	= \const{mon} (N\const{clos})^s\norm{u}{\A_s^{\rm pr}}
\end{align*}
and thus the proof. 
\end{proof}

\begin{remark}[error convergence]
With reliability (Theorem~\ref{thm:reliability}) and efficiency (Theorem~\ref{thm:efficiency}), the convergence results of this section of course also hold for the total error 
\begin{align*}
	\norm{u - u_H}{H^1_\kappa(\Omega)} + \Big(\sum_{T\in\TT_H}  h_T^2 \norm{(1-\Pi_H)(f-\kappa^2 u_H)}{T}^2\Big)^{1/2},
\end{align*}
being equivalent to the estimator $\eta_H$. 
In case of piecewise polynomial $\kappa$, a standard argument further shows that the corresponding approximation class can be decoupled into the approximation class for the error $\norm{u - u_H}{H^1_\kappa(\Omega)}$ and the oscillation term $(\sum_{T\in\TT_H}  h_T^2 \norm{(1-\Pi_H)f}{T}^2)^{1/2}$; see, e.g., \cite[Proposition~4.6]{cfpp14}. 
\end{remark}

%%%%%%%%%%%%%%%%%%%%%%%%%%%%%%%%%%%%%%%%%%%%%%%%%%%%%%%%%%%%%%%
\section{Axioms of adaptivity for the estimator}\label{sec:axioms}
%%%%%%%%%%%%%%%%%%%%%%%%%%%%%%%%%%%%%%%%%%%%%%%%%%%%%%%%%%%%%%%

Throughout this section, let $\widehat\TT_H$ be a conforming triangulation of $\Omega$ and $\widehat\TT_h$ a finer conforming triangulation of $\Omega$ in the sense that 
\begin{align*}
	T = \bigcup\set{T'\in\widehat\TT_h}{T'\subseteq T}
	\quad \forall T \in \widehat\TT_H. 
\end{align*}
Let $\TT_H\subset \widehat\TT_H$ and $\TT_h\subset \widehat\TT_h$ be finite subtriangulations forming the sets $\Omega_H$ and $\Omega_h$ such that $\Omega_H \subseteq \Omega_h$. 
In particular, this yields the nestedness
\begin{align}
	\SS_0^p(\TT_H) \subseteq \SS_0^p(\TT_h).
\end{align}
Under the assumption that 
\begin{align}
	\supp f \subseteq \overline \Omega_H, 
\end{align}
we verify (slight variations of) the four axioms of adaptivity~\cite{cfpp14} for the error estimator~\eqref{eq:estimator}. 
We will crucially use them in the following section, where we give the proof of R-linear and optimal convergence of Algorithm~\ref{alg:adaptive}.
%Together with two further properties of the employed local mesh refinement strategy, these axioms guarantee optimal convergence of a corresponding adaptive algorithm. 

\begin{lemma}[stability on non-refined elements~(A1)]\label{lem:stability}
There exists a constant $\const{stab}$ depending only on the dimension $d$,
the polynomial degree $p$, shape regularity of $\TT_h$, and $\max_{T\in\TT_h} (\kTp h_T)$
such that
\begin{align}\label{eq:stability}
|\eta_h(\UU_H,v_h) - \eta_H(\UU_H,v_H)| 
\le \const{stab} \norm{v_h - v_H}{H^1_\kappa(\Omega)}
\\
\,\,\, \forall \UU_H \subseteq \TT_H \cap \TT_h, v_H \in \SS_0^p(\TT_H), v_h \in \SS_0^p(\TT_h).
\end{align}
\end{lemma}

\begin{proof}
The result follows %as in the standard case of bounded domains, e.g., \cite{ckns08}, 
by the use of the reverse triangle inequality
\begin{align*}
	&|\eta_h(\UU_H,v_h) - \eta_H(\UU_H,v_H)|
	\\
	&\quad\le \Big(\sum_{T\in\UU_H} h_T^2 \norm{\kappa^2 (v_h - v_H) - \Delta (v_h - v_H)}{L^2(T)}^2
		+ h_T \norm{[\partial_{\bs{n}} (v_h - v_H)]}{\partial T \cap \Omega}^2 \Big)^{1/2},
\end{align*}
the inverse inequality
\begin{align*}
	h_T^2 \norm{\Delta (v_h - v_H)}{T}^2 \lesssim \norm{\nabla (v_h - v_H)}{T}^2,
\end{align*}
the trace and inverse inequality
\begin{align*}
	&h_T \norm{\jump{\partial_{\bs{n}} (v_h - v_H)}}{\partial T\cap\Omega}^2
	\\
	&\quad\lesssim \norm{\nabla (v_h - v_H)}{\omega_h^*[T]}^2 + h_T \norm{\nabla (v_h - v_H)}{\omega_h^*[T]} \norm{D^2 (v_h - v_H)}{\omega_h^*[T]} 
	\\
	&\quad\lesssim \norm{\nabla (v_h - v_H)}{\omega_h^*[T]}^2,
\end{align*}
where $\omega_h^*[T]:=\set{T'\in\TT_h}{T \text{ and }T' \text{ share a face}}$,
and a finite overlap of these patches depending only on shape regularity of $\TT_h$. 
Finally, for the $L^2(T)$ term, we have that
\begin{equation*}
h_T^2 \norm{\kappa^2 (v_h - v_H)}{L^2(T)}^2
\leq
(\kTp h_T)^2 \norm{\kappa(v_h-v_H)}{L^2(T)}^2
\leq
(\kTp h_T)^2 \norm{v_h-v_H}{H^1_{\kappa}(T)}^2.
\end{equation*}
This concludes the proof.
\end{proof}

\begin{lemma}[reduction on refined elements (A2)]\label{lem:reduction}
Suppose there exists $0<q_{\rm red}<1$ such that $h_{T'} \le q_{\rm red}^{2} h_T$ for all $T'\in\TT_h\setminus\TT_H$ with $T'\subseteq T$. 
Then, it holds that
\begin{align}\label{eq:reduction1}
	\eta_h(\TT_h \setminus \TT_H, v_H) \le q_{\rm red} \eta_H(\TT_H \setminus \TT_h, v_H)
	\quad \forall v_H \in \SS_0^p(\TT_H).
\end{align}
In particular, this yields that
\begin{align}\label{eq:reduction2}
\begin{split}
	\eta_h(\TT_h \setminus \TT_H, v_h) \le q_{\rm red} \eta_H(\TT_H \setminus \TT_h, v_H) &+ \const{stab} \norm{v_h - v_H}{H^1_\kappa(\Omega)}
	\\
	&\quad\forall v_H \in \SS_0^p(\TT_H), v_h \in \SS_0^p(\TT_h).
\end{split}
\end{align}
\end{lemma}

\begin{proof}
By definition, we have that 
\begin{align*}
	\eta_h(\TT_h \setminus \TT_H, v_H)^2 
	= \sum_{T\in\TT_h \setminus \TT_H} h_T^2 \norm{f - \kappa^2 v_H + \Delta v_H}{T}^2 + h_T \norm{[\partial_{\bs{n}} v_H]}{\partial T\cap \Omega}^2.
\end{align*}
Since $\supp f \cup \supp v_H \subseteq \overline\Omega_H$, it suffices to consider $T\in\TT_h|_{\Omega_H} := \set{T'\in\TT_h}{T'\subseteq\overline\Omega_H}$.
We set $\TT_h|_T:=\set{T'\in\TT_h}{T'\subseteq T}$ for all $T\in\TT_H$. 
Since refined elements $T\in\TT_H\setminus\TT_h = \TT_H\setminus\TT_h|_{\Omega_H}$ are the union of their children, it holds that
\begin{align*}
	&\sum_{T\in\TT_h|_{\Omega_H}\setminus\TT_H} \eta_h(T,v_H)^2 
	= \sum_{T\in\TT_H\setminus\TT_h} \sum_{T'\in\TT_h|_T} \eta_h(T',v_H)^2
	\\
	&\qquad\qquad\qquad\qquad= \sum_{T\in\TT_H\setminus\TT_h} \sum_{T'\in\TT_h|_T}  h_{T'}^2 \norm{f - \kappa^2 v_H + \Delta v_H}{T'}^2 + h_{T'} \norm{[\partial_{\bs{n}} v_H]}{\partial T'\cap \Omega}^2.
\end{align*}
The assumption $h_{T'} \le q_{\rm red}^2 h_T$ and the fact that the jump $[\partial_{\bs{n}} v_H]$ vanishes on all faces which lie inside $T$ give that
\begin{align*}
	\eta_h(\TT_h \setminus \TT_H, v_H)^2
	&\le q_{\rm red}^2 \sum_{T\in\TT_H \setminus \TT_h} \sum_{T' \in\TT_h|_T} h_T^2 \norm{f - \kappa^2 v_H + \Delta v_H}{T'}^2 + h_T \norm{[\partial_{\bs{n}} v_H]}{\partial T'\cap\Omega}^2
	\\
	& = q_{\rm red}^2 \sum_{T\in\TT_H \setminus \TT_h} \eta_H(T,v_H)^2 
	= q_{\rm red}^2 \eta_H(\TT_H\setminus\TT_h, v_H)^2. 
\end{align*}
This concludes the proof of \eqref{eq:reduction1}. 

The second assertion~\eqref{eq:reduction2} follows from stability (A1) applied to $\TT_H^{\rm (A1)} = \TT_h^{\rm (A1)} = \TT_h$ and $\UU_H^{\rm (A1)} = \TT_h \setminus\TT_H$ and \eqref{eq:reduction1}, i.e., 
\begin{align*}
	\eta_h(\TT_h \setminus \TT_H, v_h) &\le \eta_h(\TT_h \setminus \TT_H, v_H) + \const{stab} \norm{v_h - v_H}{H^1_\kappa(\Omega)}
	\\
	&\le q_{\rm red} \eta_H(\TT_H \setminus \TT_h, v_H) + \const{stab} \norm{v_h - v_H}{H^1_\kappa(\Omega)}.
\end{align*}
This concludes the proof.
\end{proof}

\begin{remark}[estimator convergence]
It is noteworthy that {\rm (A1)}--{\rm (A2)} imply estimator convergence $\lim_{\ell\to\infty}\eta_\ell = 0$ of a standard adaptive algorithm based on D\"orfler marking and bisection of elements, even though the initial computational domain $\Omega_0$ is never enlarged in the sense that $\Omega_\ell = \Omega_0$ for the sequence of generated meshes $(\TT_\ell)_{\ell\in\N_0}$. 
The elementary proof, e.g., \cite[Lemma~4.7 \& Corollary 4.8]{cfpp14}, also exploits a priori convergence of the corresponding sequence of Galerkin approximations $(u_\ell)_{\ell\in\N_0}$ to some $u_\infty \in \overline{\bigcup_{\ell\in\N_0} \SS_0^p(\TT_\ell)}$. 
For a simple model problem, such an a priori convergence result is already found in~\cite{bv84}.
The proof readily extends to an abstract Lax--Milgram setting. 
However, an adaptive algorithm remaining on $\Omega_0$ can in general not guarantee error convergence $\lim_{\ell\to\infty} \norm{u - u_\ell}{H^1_\kappa(\Omega)}$, i.e., $u = u_\infty$. 
Indeed, such an algorithm does not necessarily ensure that the minimal (weighted) mesh sizes at the artificial boundary $\min\set{\kTm h_T}{T\in\TT_\ell, T\cap \Gamma_\ell \neq \emptyset}$ corresponding to the sequence of meshes $(\TT_\ell)_{\ell\in\N_0}$ is uniformly bounded from below so that the reliability constant of Theorem~\ref{thm:reliability} might degenerate as $\ell\to\infty$.
\end{remark}

\begin{lemma}[discrete reliability (A3)]\label{lem:discrete_reliability}
Suppose that $\widehat\TT_H$ is uniformly shape regular~\eqref{eq:shape_regularity}. 
Then, there exists a constant $\const{drel}>0$ which depends only on the dimension $d$, the polynomial degree $p$, the shape regularity of $\widehat\TT_H$, and
$\min\set{\kTm h_T}{T\in\TT_H, T\cap \Gamma_H \neq \emptyset}$
such that 
\begin{subequations}\label{eq:discrete_reliability}
\begin{align}
	\norm{u_h - u_H}{H^1_\kappa(\Omega)} \le \const{drel} \eta_H(\RR_{H,h}),
\end{align}
where
\begin{align}
	\RR_{H,h} := (\TT_H \setminus \TT_h) \cup \set{T\in\TT_H\cap \TT_h}{\partial T\cap \Gamma_{H,h} \neq \emptyset}
\end{align}
\end{subequations}
with $\Gamma_{H,h} := \partial\Omega_H \setminus \partial\Omega_h$. 
Moreover, there exists a constant $\const{ref}>0$ depending only on the dimension $d$ and the shape regularity of $\widehat\TT_H$ such that
\begin{align}\label{eq:bound_R}
	\#\RR_{H,h} \le \const{ref} \, \#(\TT_H \setminus \TT_h) \le \const{ref} \, \#(\TT_h \setminus \TT_H).
\end{align}
\end{lemma}

\begin{proof}
We prove the assertion in three steps, proceeding similarly as in Theorem~\ref{thm:reliability}.

\noindent
\textbf{Step~1:}
Coercivity shows the discrete inf-sup stability
\begin{align*}
\norm{u_h - u_H}{H^1_\kappa(\Omega)}
\le
\sup_{v_h\in \SS_0^p(\TT_h)\setminus\{0\})} \frac{a(u_h - u_H, v_h)}{\norm{v_h}{H^1_\kappa(\Omega)}}.  
\end{align*}
Galerkin orthogonality gives that 
\begin{align*}
	a(u_h - u_H,v_h) = a(u_h - u_H,v_h - v_H) \quad \forall v_H \in \SS_0^p(\TT_H).
\end{align*}
Exploiting the definition of $u_h$ and $\supp f\subseteq \overline\Omega_H$, we split $a(u_h - u_H,v_h - v_H)$ as follows 
\begin{align*}
	&a(u_h - u_H,v_h - v_H) 
	\\
	&\qquad = \dual{f}{v_h - v_H}_{\Omega_H} - \dual{\kappa^2 u_H}{v_h - v_H}_{\Omega_H} - \dual{\nabla u_H}{\nabla (v_h - v_H)}_{\Omega_H}.
\end{align*}
Integration by parts shows that 
\begin{align*}
	&\dual{f}{v_h - v_H}_{\Omega_H} - \dual{\kappa^2 u_H}{v_h - v_H}_{\Omega_H} - \dual{\nabla u_H}{\nabla (v_h - v_H)}_{\Omega_H}
	\\
	&\quad = \sum_{T \in \TT_H} \dual{f}{v_h - v_H}_T - \dual{\kappa^2 u_H}{v_h - v_H}_T + \dual{\Delta u_H}{v_h - v_H}_T 
	\\
	&\qquad\qquad - \dual{\partial_{\bs{n}} u_H}{v_h - v_H}_{\partial T}
	\\
	&\quad = \sum_{T \in \TT_H} \dual{f - \kappa^2 u_H + \Delta u_H}{v_h - v_H}_T - \frac{1}{2} \dual{\jump{\partial_{\bs{n}} u_H}}{v_h - v_H}_{\partial T \cap \Omega_H} 
	\\
	&\qquad\qquad - \dual{\partial_{\bs{n}} u_H}{v_h - v_H}_{\partial T \cap \Gamma_{H,h}}.
\end{align*}
In contrast to the standard case of bounded domains, the last term $\dual{\partial_{\bs{n}} u_H}{v_h - v_H}_{\partial T \cap \Gamma_{H,h}}$ does not necessarily vanish, as $v_h - v_H$ is in general only $0$ on $\partial\Omega_h$ but not on $\partial\Omega_H$. 
Note the identities $(\partial T\cap\Omega_H) \cup (\partial T \cap \Gamma_{H,h}) = \partial T \setminus \partial \Omega_h = \partial T \cap \Omega_h \subseteq \partial T\cap\Omega$ and  $\partial_{\bs{n}} u_H = \jump{\partial_{\bs{n}} u_H}$ on $\partial T\cap\Gamma_{H,h}$. 
We apply the Cauchy--Schwarz inequality to see that 
\begin{align*}
	&\sum_{T \in \TT_H} \dual{f - \kappa^2 u_H + \Delta u_H}{v_h - v_H}_T - \frac{1}{2} \dual{\jump{\partial_{\bs{n}} u_H}}{v_h - v_H}_{\partial T \cap \Omega_H} 
	\\
	&\qquad- \dual{\partial_{\bs{n}} u_H}{v_h - v_H}_{\partial T \cap \Gamma_{H,h}}
	\\
	&\quad \le \sum_{T \in \TT_H} \Big\{ \Big(h_T^2 \norm{f - \kappa^2 u_H + \Delta u_H}{T}^2 + h_T \norm{\jump{\partial_{\bs{n}} u_H}}{\partial T\cap \Omega}^2\Big)^{1/2} 
	\\
	&\qquad \times  \Big( h_T^{-2} \norm{v_h - v_H}{T}^2 + h_T^{-1} \norm{v_h - v_H}{\partial T \cap \Omega}^2 \Big)^{1/2}\Big\}.
\end{align*}

\noindent
\textbf{Step~2:}
We conclude the proof of \eqref{eq:discrete_reliability} choosing $v_H := I_H v_h$ for some Scott--Zhang-type interpolation operator $I_H: H^1(\Omega) \to \SS_0^p(\TT_H)$
with the following local projection property
\begin{align}\label{eq:local_projection}
	\big((1-I_H) w_h\big)|_T = 0 
	\quad \forall w_h \in \SS_0^p(\TT_h), T \in \TT_H \cap \TT_h \text { with } \partial T\cap \Gamma_{H,h} = \emptyset
\end{align}
and the following local approximation property
\begin{align}\label{eq:local_approximation_sz}
	h_T^{-2} \norm{(1 - I_H) w}T^2 + h_T^{-1} \norm{(1 - I_H) w}{\partial T\cap\Omega}^2
	\lesssim \norm{w_h}{H^1_\kappa(\widehat\omega_H[T])}^2
	\quad \forall w \in H^1_0(\Omega), T\in\TT_H,
\end{align}
where $\widehat\omega_H[T] := \bigcup \set{T'\in\widehat\TT_H}{T\cap T' \neq \emptyset}$ denotes again the patch of $T$. 
Recall that the local approximation property was already required in the proof of reliability (Theorem~\ref{thm:reliability}); cf.~\eqref{eq:local_approximation}.
We construct such an operator in the following two substeps. 

%% TODO: check this carefully
\noindent
\textbf{Step~2.1:}
We consider the set of Lagrange nodes $\NN_H$ associated to the space of continuous piecewise polynomials of degree $p$ on $\TT_H$ along with the corresponding basis $\set{\phi_{H,z}}{z\in\NN_H}$.
For each node $z\in\NN_H$, let $\tau_{H,z}$ be either an element or a face in $\TT_H$ with $z\in\tau_{H,z}$ where we choose $\tau_{H,z}$ as interior face of $\TT_H$ and $\TT_h$ whenever such a face exists. 
Moreover, let $\phi_{H,z}^*\in\PP^p(\tau_{H,z})$ denote the dual basis function of $\phi_{H,z}$ with respect to the $L^2(\tau_{H,z})$ scalar product. 
Then, a suitable candidate is given by
\begin{align*}
	I_H w := \sum_{z\in\NN_H \setminus \partial\Omega_H} \int_{\tau_{H,z}} \phi_{H,z}^* w \d x \, \phi_{H,z}
	\quad \forall w \in H^1(\Omega).
\end{align*}
Let $w_h \in \SS_0^p(\TT_h)$ and let $T\in\TT_H \cap \TT_h$ with $\partial T\cap \Gamma_{H,h} = \emptyset$. 
Moreover, let $z\in\NN_T:=\NN_H \cap T=\NN_h \cap T$.
If $z$ lies in the interior of $T$ or on a $\TT_H$-interior face of $T$, then $w_h|_{\tau_{H,z}}\in\PP^p(\tau_{H,z})$ by our choice of $\tau_{H,z}$ and thus $(I_H w_h)(z) =   \int_{\tau_{H,z}} \phi_{H,z}^* w_h \d x  = w_h(z)$. 
If $z$ lies on a boundary face $F$ of $\partial\Omega_H$, then the assumption $\partial T\cap \Gamma_{H,h} = \partial T \cap (\partial\Omega_H \setminus\partial\Omega_h) = \emptyset$ guarantees that $F \subseteq \partial\Omega_h$, 
and thus $(I_H w_h)(z) = 0 = w_h(z)$ due to $z\in\partial\Omega_H$ and $w_h\in \SS_0^p(\TT_h)$. 
Since $z\in\NN_T$ was arbitrary, this shows that $(I_H w_h)|_T = w_h|_T$.

\noindent
\textbf{Step~2.2:}
A standard scaling argument shows the local $L^2$-stability
\begin{align}\label{eq:l2_stable_sz}
	\norm{I_H w}{T} \lesssim \norm{w}{L^2(\omega_H[T])} 
	\quad \forall w \in H^1(\Omega), T\in\TT_H,
\end{align}
where $\omega_H[T] := \bigcup \set{T'\in\TT_H}{T\cap T' \neq \emptyset}$. 
Let $w\in H_0^1(\Omega)$. 
If $T\cap \partial\Omega_H = \emptyset$, then $(I_H c)|_T = c$ for all $c\in\R$. 
In this case, we set $c_T:=|\widehat\omega_H[T]|^{-1} \int_{\widehat\omega_H[T]} w \d x$.
If $T\cap \partial\Omega_H\cap\partial\Omega \neq \emptyset$, we set $c_T:=0$. 
In either case, \eqref{eq:l2_stable_sz} and the Poincar\'e--Friedrichs inequality on element patches, e.g., ~\cite{cf00}, give that
\begin{align*}
	h_T^{-2}\norm{(1-I_H) w}{T}^2
	= h_T^{-2}\norm{(1-I_H) (w-c_T)}{T}^2
	\lesssim h_T^{-2}\norm{w-c_T}{\widehat\omega_H[T]}^2
	\lesssim \norm{\nabla w}{\widehat\omega_H[T]}^2.
\end{align*}
While~\eqref{eq:l2_stable_sz} is satisfied for the small patch $\omega_H[T]$, the latter argument requires $\widehat\omega_H[T]$, since, in contrast to $\omega_H[T]$, $\widehat\omega_H[T]$ is always a Lipschitz domain. 
Similarly, an inverse estimate yields that
\begin{align*}
	\norm{\nabla I_H w}{T}^2
	= \norm{\nabla I_H (w-c_T)}{T}^2
	\lesssim h_T^{-2}\norm{w-c_T}{\widehat\omega_H[T]}^2
	\lesssim \norm{\nabla w}{\widehat\omega_H[T]}^2.
\end{align*}
Together with the trace inequality, this shows~\eqref{eq:local_approximation_sz} if $T\cap \partial\Omega_H = \emptyset$ or $T\cap \partial\Omega_H \cap \partial\Omega \neq \emptyset$.
It remains to consider the non-standard case $T\cap \Gamma_H \neq \emptyset$. 
In this case, we apply the trace inequality and $\kTm h_T \gtrsim 1$,
followed by local $L^2$-stability~\eqref{eq:l2_stable_sz} and an inverse estimate to see that 
\begin{align}
\label{eq:dependency_kh}
\begin{split}
&h_T^{-2} \norm{(1 - I_H) w}T^2 + h_T^{-1} \norm{(1 - I_H) w}{\partial T\cap\Omega}^2
\\
&\qquad\qquad\lesssim (\kTm h_T)^{-2}(\kTm)^2\norm{(1-I_H) w}T^2 + \norm{\nabla (1-I_H)w}{T}^2
\lesssim \norm{w}{H^1_\kappa(\omega_H[T])}^2.
\end{split}
\end{align}
This concludes the proof of~\eqref{eq:discrete_reliability}. 

\noindent
\textbf{Step~3:}
It remains to show~\eqref{eq:bound_R}. 
Since refined elements are at least split once, it holds that 
\begin{align*}
	\#(\TT_H \setminus \TT_h)
	= \#\big(\TT_H \setminus (\TT_h|_{\Omega_H})\big)
	\le \#\TT_h|_{\Omega_H} - \#\TT_H
	\le \#\big((\TT_h|_{\Omega_H}) \setminus \TT_H\big)
	\le \#(\TT_h \setminus \TT_H),
\end{align*}
where as before $\TT_h|_{\Omega_H} := \set{T\in\TT_h}{T\subseteq \overline \Omega_H}$. 
Let $T\in\TT_H\cap \TT_h$ with $\partial T\cap \Gamma_{H,h} \neq \emptyset$.
This means that there exists $T'\in\widehat\TT_H\setminus\TT_H$ with $T\cap T' \neq \emptyset$ that is activated in $\TT_h$ in the sense that there exists $T''\in \TT_h\setminus\TT_H$ with $T''\subseteq T'$ and $T\cap T'' \neq \emptyset$.
Since the number of elements in $\widehat\TT_H$ having non-empty intersection with $T'$ is uniformly bounded depending only on the dimension $d$ and the shape regularity of $\widehat\TT_H$, we infer that 
\begin{align*}
	\# \set{T\in\TT_H\cap \TT_h}{\partial T\cap \Gamma_{H,h} \neq \emptyset} \lesssim \#(\TT_h \setminus \TT_H).
\end{align*}
Overall, we conclude \eqref{eq:bound_R} and thus the proof. 
\end{proof}

\begin{remark}[dependency on $\min \kTm h_T$]
\label{remark_kh}
For simplicity, we have established Lemma~\ref{lem:discrete_reliability} with
uniform constant~$\const{drel}$ under the assumption that
$\tau := \min\set{\kTm h_T}{T\in\TT_H, T\cap \Gamma_H \neq \emptyset}$
remains uniformly bounded from below. However, a careful inspection of
the proof shows that this assumption can be relaxed by including the
multiplicative factor $\tau^{-1}$ in~$\const{drel}$. This is, in particular,
clear from ~\eqref{eq:dependency_kh}. A similar comment applies to Theorem~\ref{thm:reliability}
and the constant $\const{rel}$. We further refer the reader to~\cite{chaumontfrelet_2025b},
where a similar analysis is explicitly carried out. 
\end{remark}

\begin{lemma}[general quasi-orthogonality (A4)] \label{lem:orthogonality}
Let $(\widehat\TT_\ell)_{\ell\in\N_0}$ be a sequence of conforming triangulations of $\Omega$ and $(\TT_\ell)_{\ell\in\N_0}$ a sequence of corresponding subtriangulations, i.e., $\TT_\ell\subseteq\widehat\TT_\ell$ for all $\ell\in\N_0$, such that $\widehat\TT_{\ell+1}$ is finer than $\widehat\TT_\ell$ and $\Omega_\ell\subseteq \Omega_{\ell+1}$ for all $\ell\in\N_0$. 
Suppose that $(\widehat\TT_\ell)_{\ell\in\N_0}$ is uniformly shape regular and the element sizes $h_T$ for $T\in\widehat\TT_\ell\setminus\TT_\ell$ are uniformly bounded from below (cf.~\eqref{eq:shape_regularity} and \eqref{eq:infsup_h}). 
%\begin{align}\label{eq:infsup_h}
%	\const{inf} \le \inf_{T\in\widehat\TT_\ell\setminus\TT_\ell} h_T \quad \forall \ell\in\N_0.
%\end{align}
Moreover, suppose that
\begin{align}
	\supp f \subseteq \overline \Omega_0.
\end{align}
Then, Theorem~\ref{thm:reliability} and Remark~\ref{rem:extended_estimator} yield that the estimators $(\eta_\ell)_{\ell\in\N_0}$ are reliable with uniform constant $\const{rel}>0$, and it holds that 
\begin{align}\label{eq:orthogonality}
	\sum_{k = \ell}^\infty \norm{u_{k+1} - u_k}{H^1_\kappa(\Omega)}^2
	\le \const{rel}^2 \eta_\ell^2
	\quad \forall \ell\in\N_0.
\end{align}
% for $\const{orth} = \cancel{\max\{\norm{\kappa^{-2}}{L^\infty(\Omega)},1\}\max\{\norm{\kappa^{2}}{L^\infty(\Omega)},1\}} \const{rel}^2$. 
\end{lemma}

\begin{proof}
The proof follows from Galerkin orthogonality as well as reliability (Theorem~\ref{thm:reliability}):
\begin{equation*}
	\sum_{k = \ell}^\infty \norm{u_{k+1} - u_k}{H^1_\kappa(\Omega)}^2
	= \sum_{k = \ell}^\infty \norm{u - u_k}{H^1_\kappa(\Omega)}^2 - \norm{u - u_{k+1}}{H^1_{\kappa}(\Omega)}^2
	\le \norm{u - u_\ell}{H^1_\kappa(\Omega)}^2
	\le \const{rel}^2 \eta_\ell^2.
\end{equation*}
This concludes the proof.
\end{proof}

\begin{remark}[non-symmetric bilinear form] \label{rem:nonsymmetric}
If the considered PDE also contains an advection term, the resulting bilinear $a(\cdot,\cdot)$ is no longer symmetric, and the proof of general quasi-orthogonality {\rm(A4)} does not carry over. 
On bounded domains, general quasi-orthogonality is proved in~\cite{ffp14} using a compactness argument in this case. 
However, on unbounded domains, lower-order terms are \emph{not} compact perturbations with respect to $H_0^1(\Omega)$. 
Instead, the analysis of \cite{feischl22} is applicable, provided that $a(\cdot,\cdot)$ is coercive. 
For some $\delta>0$, this yields that
\begin{align}
	\sum_{k = \ell}^{\ell+N} \norm{u_{k+1} - u_k}{H^1(\Omega)}^2
	\lesssim N^{1-\delta} \norm{u - u_\ell}{H^1(\Omega)}^2
	\lesssim N^{1-\delta} \eta_\ell^2
	\quad \forall \ell,N\in\N_0,
\end{align}
which is sufficient to derive R-linear and thus optimal convergence; see \cite{feischl22} and Section~\ref{sec:proofs} for details.
\end{remark}

%%%%%%%%%%%%%%%%%%%%%%%%%%%%%%%%%%%%%%%%%%%%%%%%%%%%%%%%%%%%%%%
\section{Proof of R-linear and optimal convergence} \label{sec:proofs}
%%%%%%%%%%%%%%%%%%%%%%%%%%%%%%%%%%%%%%%%%%%%%%%%%%%%%%%%%%%%%%%

\begin{proof}[Proof of R-linear convergence (Theorem~\ref{thm:linear_convergence})]
Algorithm~\ref{alg:adaptive} particularly guarantees that all marked elements are refined, i.e., $\MM_\ell \subseteq \TT_\ell \setminus \TT_{\ell+1}$. 
Moreover, refined elements are at least bisected once. 
Together with D\"orfler marking and stability~(A1) and reduction~(A2) with $q_{\rm red} = 2^{-1/(2d)}$ proven in Lemma~\ref{lem:stability} and Lemma~\ref{lem:reduction}, a standard argument yields for all (sufficiently small) $\delta>0$ the estimator reduction
\begin{align*}
	\eta_{\ell+1}^2 %&= \eta_{\ell+1}(\TT_{\ell+1}\setminus\TT_\ell)^2 + \eta_{\ell+1}(\TT_{\ell+1}\cap\TT_\ell)^2 
%	\\
%	&\le (1+\delta) q_{\rm red}^2 \eta_\ell(\TT_\ell\setminus\TT_{\ell+1})^2 + (1+\delta) \underbrace{\eta_\ell(\TT_{\ell+1}\cap\TT_\ell)^2}_{=\eta_\ell^2 - \eta_\ell(\TT_\ell \setminus \TT_{\ell+1})^2} + 2(1+\delta^{-1}) \const{stab}^2\norm{u_{\ell+1} - u_\ell}{H^1(\Omega)}^2 
%	\\
%	&\le (1+\delta) \eta_\ell^2  + (1+\delta)(q_{\rm red}^2-1) \underbrace{\eta_\ell(\TT_\ell\setminus\TT_{\ell+1})^2}_{\ge \eta_\ell(\MM_\ell)^2 \ge \theta \eta_\ell^2}  + 2(1+\delta^{-1}) \const{stab}^2\norm{u_{\ell+1} - u_\ell}{H^1(\Omega)}^2
%	\\
%	& \le (1+\delta) \eta_\ell^2  + (1+\delta)(q_{\rm red}^2-1)\theta \eta_\ell^2  + 2(1+\delta^{-1}) \const{stab}^2\norm{u_{\ell+1} - u_\ell}{H^1(\Omega)}^2
%	\\
	& \le  (1+\delta) \big(1- (1-q_{\rm red}^2)\theta\big)  \eta_\ell^2  + 2(1+\delta^{-1}) \const{stab}^2\norm{u_{\ell+1} - u_\ell}{H^1_\kappa(\Omega)}^2;
\end{align*}
see, e.g., \cite[Lemma~4.7]{cfpp14}.
Combining this with general quasi-orthogonality (A4) proven in Lemma~\ref{lem:orthogonality}, elementary calculus gives the desired result; see, e.g., \cite[Proposition~4.10]{cfpp14}. 
Regarding the dependence of the constants involved, we note that $\min\set{\kTm h_T}{T\in\TT_H, T\cap \Gamma_H \neq \emptyset}$ for $\TT_H\in\T_H$ is uniformly bounded from below
in terms of $\inf\set{\kappa_{\widehat\omega_0[T]}^- h_T}{T\in \widehat{\TT}_0\setminus \TT_0}$ and the shape regularity of $\widehat\TT_0$.
\end{proof}

\begin{proof}[Proof of optimal convergence (Theorem~\ref{thm:optimal_convergence})]
We start with the closure estimate from Proposition~\ref{prop:closure}, i.e., 
\begin{align*}
	\# (\TT_\ell \setminus \TT_0)
	=\# (\widehat \TT_\ell \setminus \widehat \TT_0)
	\le \const{clos} \sum_{j=0}^{\ell-1} \# \MM_j
	\quad \forall \ell\in\N_0.
\end{align*}
Now, we estimate $\# \MM_j$ for $j\in\N_0$. 
Without loss of generality, we may assume that $\eta_j>0$, as then $\MM_j = \emptyset$. 
By definition of the approximation class, for any $0<q<1$, there exists $\TT_H \in \T$ with 
\begin{align*}
	\eta_H &\le \const{mon}^{-1} q \eta_j,
	\\
	\#(\TT_H \setminus \TT_0) &\le \norm{u}{\A_s}^{1/s} \const{mon}^{1/s} q^{-1/s} \eta_j^{-1/s}.
\end{align*}
Quasi-monotonicity of the estimator~\eqref{eq:quasi-monotonicity} and the overlay estimate from Proposition~\ref{prop:overlay} show for $\widehat\TT_{j_+} := \widehat\TT_H \oplus \widehat\TT_j$ that
\begin{align}\label{eq:contraction}
	\eta_{j_+} \le \const{mon} \eta_H \le q \eta_j,
\end{align}
and
\begin{align*}
	\#(\TT_{j_+} \setminus \TT_j) = \#(\widehat \TT_{j_+} \setminus \widehat \TT_j) 
	\le 2\# (\widehat \TT_H \setminus \widehat \TT_0)
	= 2\# (\TT_H \setminus \TT_0)
	 &\le 2\norm{u}{\A_s}^{1/s} \const{mon}^{1/s} q^{-1/s} \eta_j^{-1/s}.
\end{align*}

Choosing $0<q<1$ sufficiently small, \eqref{eq:contraction} together with stability (A1) from Lemma~\ref{lem:stability} and discrete reliability (A3) from Lemma~\ref{lem:discrete_reliability} readily yields D\"orfler marking for the set $\RR_{j,{j_+}}\supseteq \TT_j \setminus \TT_{j_+}$ from (A3), i.e., 
\begin{align*}
	\theta \eta_j^2 \le \eta_j(\RR_{j,{j_+}})^2,
\end{align*}
which is known as optimality of D\"orfler marking in the literature; see, e.g., \cite[Proposition~4.12]{cfpp14}. 
Minimality of $\MM_j$ and \eqref{eq:bound_R} hence imply that
\begin{align*}
	\# \MM_j \le \# \RR_{j,{j_+}} \le \const{ref} \, \#(\TT_{j_+}\setminus\TT_j). 
\end{align*}

Overall, we conclude that 
\begin{align*}
	\# (\TT_\ell \setminus \TT_0)
	=\# (\widehat \TT_\ell \setminus \widehat \TT_0)
	&\le \const{clos} \sum_{j=0}^{\ell-1} \# \MM_j
	\le \const{clos} \const{ref} \sum_{j=0}^{\ell-1} \#(\TT_{j_+} \setminus \TT_j) %+ \# \set{T\in\TT_j\cap \TT_{j_+}}{|\partial T\cap \Gamma_{j,{j_+}}| > 0}. 
	\\
	&\le 2\norm{u}{\A_s}^{1/s} \const{clos}\const{ref} \const{mon}^{1/s} q^{-1/s} \sum_{j=0}^{\ell-1} \eta_j^{-1/s}
	\\
	&\le 2\norm{u}{\A_s}^{1/s} \const{clos}\const{ref} \const{mon}^{1/s} q^{-1/s} \const{lin}^{1/s} (1 - q_{\rm lin}^{1/s})^{-1} \eta_\ell^{-1/s}
	\,\,\forall \ell\in\N_0,
\end{align*}
where the last inequality follows from linear convergence~\eqref{eq:linear_convergence}; see, e.g., \cite[Lemma~4.9]{cfpp14} for the elementary argument. 
Simple recasting concludes the proof.
\end{proof}

%%%%%%%%%%%%%%%%%%%%%%%%%%%%%%%%%%%%%%%%%%%%%%%%%%%%%%%%%%%%%%%
\section{Numerical experiments} \label{sec:numerics}
%%%%%%%%%%%%%%%%%%%%%%%%%%%%%%%%%%%%%%%%%%%%%%%%%%%%%%%%%%%%%%%

We now present a set of numerical examples that illustrate
the theory developed above. Our goal is in particular to
observe the optimal convergence rates predicted in
Theorem~\ref{thm:optimal_convergence}, as well as the reliability
and efficiency properties of the estimator established
in Theorems~\ref{thm:reliability} and~\ref{thm:efficiency}.

\subsection{Settings}

For all examples, we employ Algorithm~\ref{alg:adaptive}, where
we describe the generic setting in the following.

\subsubsection{Meshes}

In all our examples, the initial mesh $\widehat\TT_0$ is based on a Cartesian grid,
where each square is subdivided into $4$ congruent triangles. 
We will examine different mesh sizes below.
The origin is always a vertex of the Cartesian grid, and we select the longest
edge of each triangle (the one directly originating from the Cartesian
grid) as reference edge for refinement. Note that this indeed
leads to an admissible triangulation.

\subsubsection{D\"orfler marking}

We employ D\"orfler marking with $\theta = 0.2$.

\subsubsection{Iteration counts}

In all examples, we let the adaptive loop run for $100$ iterations (up to $\ell=100$).

\subsection{Smoothed fundamental solution}

We first examine the case where $\Omega = \mathbb{R}^2$
and $\kappa$ is constant. We consider a smooth radial solution given by
\begin{equation*}
u(x) = \chi(|x|)K_0(\kappa|x|),
\end{equation*}
where $K_0$ is the modified Bessel function of the second kind (i.e.,
the fundamental solution of the PDE),
and $\chi: \mathbb R_+ \to \mathbb R_+$ is a smooth transition
function such that $\chi(r) = 0$ for $r \leq 0.1$ and $\chi(r) = 1$
for $r \geq 0.9$. In the transition region $0.1 < r < 0.9$, $\chi$ is
defined as the unique polynomial of degree $7$ that ensures $C^3$
smoothness. The corresponding right-hand side $f$ is supported in
the annulus $0.1 \leq |x| \leq 0.9$. We consider three values of $\kappa^2$,
namely, $1$, $0.1$ and $0.01$.

For this benchmark, no localized refinements are required to obtain optimal
convergence rates. Nevertheless, this test case is of interest as it assesses
whether the adaptive algorithm balances the internal refinements and expansion of
the computational domain $\Omega_H$.

The initial mesh $\TT_0$ consists of the $4$ squares touching the origin,
each subdivided into 4 triangles. We examine different mesh sizes, namely
$h_0 = 1$, $4$ and $8$.

Figures~\ref{figure_smooth_P0},~\ref{figure_smooth_P1}, and~\ref{figure_smooth_P2}
respectively display the convergence histories for $p=1$, $2$, and $3$. 
In all cases, we observe the optimal convergence rate for the error in terms of the number
of degrees of freedom in the finite element space. For small values of $\kappa$
and small initial mesh size, we observe a pre-asymptotic regime that is more
pronounced for higher polynomial degrees. This is in accordance with our
requirements on the initial triangulation, and the $p$-explicit analysis
carried out in~\cite{chaumontfrelet_2025b} for a flux-equilibrated estimator.
This pre-asymptotic regime also matches the intuition: if many (unnecessarily) small
elements need  be added to push the artificial boundary $\Gamma_H$, this expansion operation
introduces degrees of freedom that are not immediately useful.

\begin{figure}
\begin{center}
\begin{minipage}{.45\linewidth}
\begin{tikzpicture}[scale=0.95]
\begin{axis}
[
	xlabel = {$N_{\rm dofs}$},
	ylabel = {$\|u-u_\ell\|_{H^1_\kappa(\Omega)}$},
	xmode = log,
	ymode = log,
	width = \linewidth
]

\plot[ultra thick,solid ,color=black,mark=none] table[x=dof,y=err] {data/fond/P0_mu_1.00_hini_008.txt};
\plot[ultra thick,dashed,color=blue ,mark=none] table[x=dof,y=err] {data/fond/P0_mu_1.00_hini_004.txt};
\plot[ultra thick,dotted,color=red  ,mark=none] table[x=dof,y=err] {data/fond/P0_mu_1.00_hini_001.txt};

\node[anchor=north east] at (rel axis cs: 0.9,0.9) {$\kappa^2 = 1$};

\plot[dashed,domain=1.e4:8.e6] {5.*x^(-0.5)};
\SlopeTriangle{0.5}{-0.15}{0.25}{-0.5}{$N_{\rm dofs}^{-1/2}$}{}

\end{axis}
\end{tikzpicture}
\end{minipage}
\begin{minipage}{.45\linewidth}
\begin{tikzpicture}[scale=0.95]
\begin{axis}
[
	xlabel = {$N_{\rm dofs}$},
	ylabel = {$\|u-u_\ell\|_{H^1_\kappa(\Omega)}$},
	xmode = log,
	ymode = log,
	width = \linewidth
]

\plot[ultra thick,solid ,color=black,mark=none] table[x=dof,y=err] {data/fond/P0_mu_0.10_hini_008.txt};
\plot[ultra thick,dashed,color=blue ,mark=none] table[x=dof,y=err] {data/fond/P0_mu_0.10_hini_004.txt};
\plot[ultra thick,dotted,color=red  ,mark=none] table[x=dof,y=err] {data/fond/P0_mu_0.10_hini_001.txt};

\node[anchor=north east] at (rel axis cs: 0.9,0.9) {$\kappa^2 = 0.1$};

\plot[dashed,domain=1.e4:5.e6] {10.*x^(-0.5)};
\SlopeTriangle{0.5}{-0.15}{0.25}{-0.5}{$N_{\rm dofs}^{-1/2}$}{}

\end{axis}
\end{tikzpicture}
\end{minipage}

\begin{minipage}{.45\linewidth}
\begin{tikzpicture}[scale=0.95]
\begin{axis}
[
	xlabel = {$N_{\rm dofs}$},
	ylabel = {$\|u-u_\ell\|_{H^1_\kappa(\Omega)}$},
	xmode = log,
	ymode = log,
	width = \linewidth
]

\plot[ultra thick,solid ,color=black,mark=none] table[x=dof,y=err] {data/fond/P0_mu_0.01_hini_008.txt};
\plot[ultra thick,dashed,color=blue ,mark=none] table[x=dof,y=err] {data/fond/P0_mu_0.01_hini_004.txt};
\plot[ultra thick,dotted,color=red  ,mark=none] table[x=dof,y=err] {data/fond/P0_mu_0.01_hini_001.txt};

\node[anchor=north east] at (rel axis cs: 0.9,0.9) {$\kappa^2 = 0.01$};

\plot[dashed,domain=5.e3:2.e6] {20.*x^(-0.5)};
\SlopeTriangle{0.5}{-0.15}{0.25}{-0.5}{$N_{\rm dofs}^{-1/2}$}{}

\end{axis}
\end{tikzpicture}
\end{minipage}
\begin{minipage}{.45\linewidth}
\hspace{3cm}
\begin{tikzpicture}[scale=0.95]
\draw[ultra thick,solid ,color=black,mark=none] (0,0) -- (1,0); \draw (1,0) node[anchor=west] {$h_0 = 8$};
\draw[ultra thick,dashed,color=blue ,mark=none] (0,1) -- (1,1); \draw (1,1) node[anchor=west] {$h_0 = 4$};
\draw[ultra thick,dotted,color=red  ,mark=none] (0,2) -- (1,2); \draw (1,2) node[anchor=west] {$h_0 = 1$};
\end{tikzpicture}
\end{minipage}
\caption{Convergence history for the smoothed fundamental solution and $p=1$.}
\label{figure_smooth_P0}
\end{center}
\end{figure}
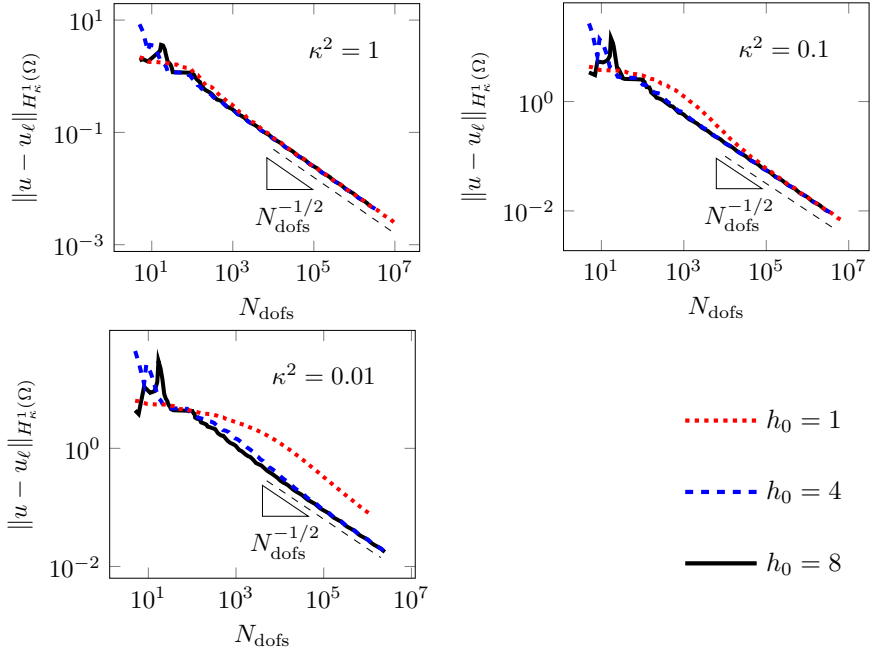
\begin{figure}
\begin{center}
\begin{minipage}{.45\linewidth}
\begin{tikzpicture}[scale=0.95]
\begin{axis}
[
	xlabel = {$N_{\rm dofs}$},
	ylabel = {$\|u-u_\ell\|_{H^1_\kappa(\Omega)}$},
	xmode = log,
	ymode = log,
	width = \linewidth
]

\plot[ultra thick,solid ,color=black,mark=none] table[x=dof,y=err] {data/fond/P1_mu_1.00_hini_008.txt};
\plot[ultra thick,dashed,color=blue ,mark=none] table[x=dof,y=err] {data/fond/P1_mu_1.00_hini_004.txt};
\plot[ultra thick,dotted,color=red  ,mark=none] table[x=dof,y=err] {data/fond/P1_mu_1.00_hini_001.txt};

\node[anchor=north east] at (rel axis cs: 0.9,0.9) {$\kappa^2 = 1$};

\plot[dashed,domain=5.e3:2.e6] {40*x^(-1.0)};
\SlopeTriangle{0.5}{-0.15}{0.25}{-1.0}{$N_{\rm dofs}^{-1}$}{}

\end{axis}
\end{tikzpicture}
\end{minipage}
\begin{minipage}{.45\linewidth}
\begin{tikzpicture}[scale=0.95]
\begin{axis}
[
	xlabel = {$N_{\rm dofs}$},
	ylabel = {$\|u-u_\ell\|_{H^1_\kappa(\Omega)}$},
	xmode = log,
	ymode = log,
	width = \linewidth
]

\plot[ultra thick,solid ,color=black,mark=none] table[x=dof,y=err] {data/fond/P1_mu_0.10_hini_008.txt};
\plot[ultra thick,dashed,color=blue ,mark=none] table[x=dof,y=err] {data/fond/P1_mu_0.10_hini_004.txt};
\plot[ultra thick,dotted,color=red  ,mark=none] table[x=dof,y=err] {data/fond/P1_mu_0.10_hini_001.txt};

\node[anchor=north east] at (rel axis cs: 0.9,0.9) {$\kappa^2 = 0.1$};

\plot[dashed,domain=5.e3:1.e6] {50*x^(-1.0)};
\SlopeTriangle{0.5}{-0.15}{0.2}{-1.0}{$N_{\rm dofs}^{-1}$}{}

\end{axis}
\end{tikzpicture}
\end{minipage}

\begin{minipage}{.45\linewidth}
\begin{tikzpicture}[scale=0.95]
\begin{axis}
[
	xlabel = {$N_{\rm dofs}$},
	ylabel = {$\|u-u_\ell\|_{H^1_\kappa(\Omega)}$},
	xmode = log,
	ymode = log,
	width = \linewidth
] \plot[ultra thick,solid ,color=black,mark=none] table[x=dof,y=err] {data/fond/P1_mu_0.01_hini_008.txt};
\plot[ultra thick,dashed,color=blue ,mark=none] table[x=dof,y=err] {data/fond/P1_mu_0.01_hini_004.txt};
\plot[ultra thick,dotted,color=red  ,mark=none] table[x=dof,y=err] {data/fond/P1_mu_0.01_hini_001.txt};

\node[anchor=north east] at (rel axis cs: 0.9,0.9) {$\kappa^2 = 0.01$};

\plot[dashed,domain=5.e3:5.e5] {100*x^(-1.0)};
\SlopeTriangle{0.525}{-0.15}{0.2}{-1.0}{$N_{\rm dofs}^{-1}$}{}

\end{axis}
\end{tikzpicture}
\end{minipage}
\begin{minipage}{.45\linewidth}
\hspace{3cm}
\begin{tikzpicture}[scale=0.95]
\draw[ultra thick,solid ,color=black,mark=none] (0,0) -- (1,0); \draw (1,0) node[anchor=west] {$h_0 = 8$};
\draw[ultra thick,dashed,color=blue ,mark=none] (0,1) -- (1,1); \draw (1,1) node[anchor=west] {$h_0 = 4$};
\draw[ultra thick,dotted,color=red  ,mark=none] (0,2) -- (1,2); \draw (1,2) node[anchor=west] {$h_0 = 1$};
\end{tikzpicture}
\end{minipage}
\caption{Convergence history for the smoothed fundamental solution and $p=2$.}
\label{figure_smooth_P1}
\end{center}
\end{figure}
\begin{figure}
\begin{center}
\begin{minipage}{.45\linewidth}
\begin{tikzpicture}[scale=0.95]
\begin{axis}
[
	xlabel = {$N_{\rm dofs}$},
	ylabel = {$\|u-u_\ell\|_{H^1_\kappa(\Omega)}$},
	xmode = log,
	ymode = log,
	width = \linewidth
]

\plot[ultra thick,solid ,color=black,mark=none] table[x=dof,y=err] {data/fond/P2_mu_1.00_hini_008.txt};
\plot[ultra thick,dashed,color=blue ,mark=none] table[x=dof,y=err] {data/fond/P2_mu_1.00_hini_004.txt};
\plot[ultra thick,dotted,color=red  ,mark=none] table[x=dof,y=err] {data/fond/P2_mu_1.00_hini_001.txt};

\node[anchor=north east] at (rel axis cs: 0.9,0.9) {$\kappa^2 = 1$};

\plot[dashed,domain=1.e3:5.e5] {500*x^(-1.5)};
\SlopeTriangle{0.5}{-0.15}{0.275}{-1.5}{$N_{\rm dofs}^{-3/2}$}{}

\end{axis}
\end{tikzpicture}
\end{minipage}
\begin{minipage}{.45\linewidth}
\begin{tikzpicture}[scale=0.95]
\begin{axis}
[
	xlabel = {$N_{\rm dofs}$},
	ylabel = {$\|u-u_\ell\|_{H^1_\kappa(\Omega)}$},
	xmode = log,
	ymode = log,
	width = \linewidth
]

\plot[ultra thick,solid ,color=black,mark=none] table[x=dof,y=err] {data/fond/P2_mu_0.10_hini_008.txt};
\plot[ultra thick,dashed,color=blue ,mark=none] table[x=dof,y=err] {data/fond/P2_mu_0.10_hini_004.txt};
\plot[ultra thick,dotted,color=red  ,mark=none] table[x=dof,y=err] {data/fond/P2_mu_0.10_hini_001.txt};

\node[anchor=north east] at (rel axis cs: 0.9,0.9) {$\kappa^2 = 0.1$};

\plot[dashed,domain=1.e3:5.e5] {1000*x^(-1.5)};
\SlopeTriangle{0.5}{-0.15}{0.275}{-1.5}{$N_{\rm dofs}^{-3/2}$}{}

\end{axis}
\end{tikzpicture}
\end{minipage}

\begin{minipage}{.45\linewidth}
\begin{tikzpicture}[scale=0.95]
\begin{axis}
[
	xlabel = {$N_{\rm dofs}$},
	ylabel = {$\|u-u_\ell\|_{H^1_\kappa(\Omega)}$},
	xmode = log,
	ymode = log,
	width = \linewidth
]

\plot[ultra thick,solid ,color=black,mark=none] table[x=dof,y=err] {data/fond/P2_mu_0.01_hini_008.txt};
\plot[ultra thick,dashed,color=blue ,mark=none] table[x=dof,y=err] {data/fond/P2_mu_0.01_hini_004.txt};
\plot[ultra thick,dotted,color=red  ,mark=none] table[x=dof,y=err] {data/fond/P2_mu_0.01_hini_001.txt};

\node[anchor=north east] at (rel axis cs: 0.9,0.9) {$\kappa^2 = 0.01$};

\plot[dashed,domain=1.e3:3.e5] {5000*x^(-1.5)};
\SlopeTriangle{0.5}{-0.15}{0.275}{-1.5}{$N_{\rm dofs}^{-3/2}$}{}

\end{axis}
\end{tikzpicture}
\end{minipage}
\begin{minipage}{.45\linewidth}
\hspace{3cm}
\begin{tikzpicture}[scale=0.95]
\draw[ultra thick,solid ,color=black,mark=none] (0,0) -- (1,0); \draw (1,0) node[anchor=west] {$h_0 = 8$};
\draw[ultra thick,dashed,color=blue ,mark=none] (0,1) -- (1,1); \draw (1,1) node[anchor=west] {$h_0 = 4$};
\draw[ultra thick,dotted,color=red  ,mark=none] (0,2) -- (1,2); \draw (1,2) node[anchor=west] {$h_0 = 1$};
\end{tikzpicture}
\end{minipage}
\caption{Convergence history for the smoothed fundamental solution and $p=3$.}
\label{figure_smooth_P2}
\end{center}
\end{figure}

Let us note that, for the coarsest initial mesh, the error does not decrease
mono\-ton\-ous\-ly as it should (since the finite element spaces are nested). This
behavior occurs only in the very first few iterations and is an artifact of numerical quadrature.
Indeed, although we do employ quadrature schemes of higher-order than the expected
convergence rate, they can still induce noticeable errors on very coarse meshes.

Figures~\ref{figure_smooth_P0_eff},~\ref{figure_smooth_P1_eff}, and~\ref{figure_smooth_P2_eff},
display the effectivity indices. As expected, the ratio stabilizes to a fixed value for
fine meshes. This asymptotic value increases with $p$, which is also natural for residual-based
estimators (note that, in addition, we did not attempt to obtain optimal $p$-scalings). 
We finally note that there is a preasymptotic regime where the error is ``underestimated'' as compared to
the asymptotic ratio. This preasymptotic range is more pronounced for small values of
$\kappa h_0$, which is in agreement with the theory
(as explained in Remark~\ref{remark_kh}, the constants
$\const{rel}$ and $\const{drel}$ contain a factor $(\kappa h_0)^{-1}$).

Similar to our comment above, the very low effectivity indices observed in the first few
iterations for $p=1$ are due to the use of numerical integration.

\begin{figure}
\begin{center}
\begin{minipage}{.45\linewidth}
\begin{tikzpicture}[scale=0.95]
\begin{axis}
[
	xlabel = {$N_{\rm dofs}$},
	ylabel = {$\eta_\ell/\|u-u_\ell\|_{H^1_\kappa(\Omega)}$},
	xmode = log,
	ymode = log,
	ymin  = 1,
	width = \linewidth
]

\plot[ultra thick,solid ,color=black,mark=none] table[x=dof,y expr=\thisrow{est}/\thisrow{err}] {data/fond/P0_mu_1.00_hini_008.txt};
\plot[ultra thick,dashed,color=blue ,mark=none] table[x=dof,y expr=\thisrow{est}/\thisrow{err}] {data/fond/P0_mu_1.00_hini_004.txt};
\plot[ultra thick,dotted,color=red  ,mark=none] table[x=dof,y expr=\thisrow{est}/\thisrow{err}] {data/fond/P0_mu_1.00_hini_001.txt};

\node[anchor=north east] at (rel axis cs: 0.9,0.9) {$\kappa^2 = 1$};

\end{axis}
\end{tikzpicture}
\end{minipage}
\begin{minipage}{.45\linewidth}
\begin{tikzpicture}[scale=0.95]
\begin{axis}
[
	xlabel = {$N_{\rm dofs}$},
	ylabel = {$\eta_\ell/\|u-u_\ell\|_{H^1_\kappa(\Omega)}$},
	xmode = log,
	ymode = log,
	ymin  = 1,
	width = \linewidth
]

\plot[ultra thick,solid ,color=black,mark=none] table[x=dof,y expr=\thisrow{est}/\thisrow{err}] {data/fond/P0_mu_0.10_hini_008.txt};
\plot[ultra thick,dashed,color=blue ,mark=none] table[x=dof,y expr=\thisrow{est}/\thisrow{err}] {data/fond/P0_mu_0.10_hini_004.txt};
\plot[ultra thick,dotted,color=red  ,mark=none] table[x=dof,y expr=\thisrow{est}/\thisrow{err}] {data/fond/P0_mu_0.10_hini_001.txt};

\node[anchor=north east] at (rel axis cs: 0.9,0.9) {$\kappa^2 = 0.1$};

\end{axis}
\end{tikzpicture}
\end{minipage}

\begin{minipage}{.45\linewidth}
\begin{tikzpicture}[scale=0.95]
\begin{axis}
[
	xlabel = {$N_{\rm dofs}$},
	ylabel = {$\eta_\ell/\|u-u_\ell\|_{H^1_\kappa(\Omega)}$},
	xmode = log,
	ymode = log,
	ymin  = 1,
	width = \linewidth
]

\plot[ultra thick,solid ,color=black,mark=none] table[x=dof,y expr=\thisrow{est}/\thisrow{err}] {data/fond/P0_mu_0.01_hini_008.txt};
\plot[ultra thick,dashed,color=blue ,mark=none] table[x=dof,y expr=\thisrow{est}/\thisrow{err}] {data/fond/P0_mu_0.01_hini_004.txt};
\plot[ultra thick,dotted,color=red  ,mark=none] table[x=dof,y expr=\thisrow{est}/\thisrow{err}] {data/fond/P0_mu_0.01_hini_001.txt};

\node[anchor=north east] at (rel axis cs: 0.9,0.9) {$\kappa^2 = 0.01$};

\end{axis}
\end{tikzpicture}
\end{minipage}
\begin{minipage}{.45\linewidth}
\hspace{3cm}
\begin{tikzpicture}[scale=0.95]
\draw[ultra thick,solid ,color=black,mark=none] (0,0) -- (1,0); \draw (1,0) node[anchor=west] {$h_0 = 8$};
\draw[ultra thick,dashed,color=blue ,mark=none] (0,1) -- (1,1); \draw (1,1) node[anchor=west] {$h_0 = 4$};
\draw[ultra thick,dotted,color=red  ,mark=none] (0,2) -- (1,2); \draw (1,2) node[anchor=west] {$h_0 = 1$};
\end{tikzpicture}
\end{minipage}
\caption{Effectivity indices for the smoothed fundamental solution and $p=1$.}
\label{figure_smooth_P0_eff}
\end{center}
\end{figure}
\begin{figure}
\begin{center}
\begin{minipage}{.45\linewidth}
\begin{tikzpicture}[scale=0.95]
\begin{axis}
[
	xlabel = {$N_{\rm dofs}$},
	ylabel = {$\eta_\ell/\|u-u_\ell\|_{H^1_\kappa(\Omega)}$},
	xmode = log,
	ymode = log,
	width = \linewidth
]

\plot[ultra thick,solid ,color=black,mark=none] table[x=dof,y expr=\thisrow{est}/\thisrow{err}] {data/fond/P1_mu_1.00_hini_008.txt};
\plot[ultra thick,dashed,color=blue ,mark=none] table[x=dof,y expr=\thisrow{est}/\thisrow{err}] {data/fond/P1_mu_1.00_hini_004.txt};
\plot[ultra thick,dotted,color=red  ,mark=none] table[x=dof,y expr=\thisrow{est}/\thisrow{err}] {data/fond/P1_mu_1.00_hini_001.txt};

\node[anchor=north east] at (rel axis cs: 0.9,0.9) {$\kappa^2 = 1$};

\end{axis}
\end{tikzpicture}
\end{minipage}
\begin{minipage}{.45\linewidth}
\begin{tikzpicture}[scale=0.95]
\begin{axis}
[
	xlabel = {$N_{\rm dofs}$},
	ylabel = {$\eta_\ell/\|u-u_\ell\|_{H^1_\kappa(\Omega)}$},
	xmode = log,
	ymode = log,
	width = \linewidth
]

\plot[ultra thick,solid ,color=black,mark=none] table[x=dof,y expr=\thisrow{est}/\thisrow{err}] {data/fond/P1_mu_0.10_hini_008.txt};
\plot[ultra thick,dashed,color=blue ,mark=none] table[x=dof,y expr=\thisrow{est}/\thisrow{err}] {data/fond/P1_mu_0.10_hini_004.txt};
\plot[ultra thick,dotted,color=red  ,mark=none] table[x=dof,y expr=\thisrow{est}/\thisrow{err}] {data/fond/P1_mu_0.10_hini_001.txt};

\node[anchor=north east] at (rel axis cs: 0.9,0.9) {$\kappa^2 = 0.1$};

\end{axis}
\end{tikzpicture}
\end{minipage}

\begin{minipage}{.45\linewidth}
\begin{tikzpicture}[scale=0.95]
\begin{axis}
[
	xlabel = {$N_{\rm dofs}$},
	ylabel = {$\eta_\ell/\|u-u_\ell\|_{H^1_\kappa(\Omega)}$},
	xmode = log,
	ymode = log,
	width = \linewidth
]

\plot[ultra thick,solid ,color=black,mark=none] table[x=dof,y expr=\thisrow{est}/\thisrow{err}] {data/fond/P1_mu_0.01_hini_008.txt};
\plot[ultra thick,dashed,color=blue ,mark=none] table[x=dof,y expr=\thisrow{est}/\thisrow{err}] {data/fond/P1_mu_0.01_hini_004.txt};
\plot[ultra thick,dotted,color=red  ,mark=none] table[x=dof,y expr=\thisrow{est}/\thisrow{err}] {data/fond/P1_mu_0.01_hini_001.txt};

\node[anchor=north east] at (rel axis cs: 0.9,0.9) {$\kappa^2 = 0.01$};

\end{axis}
\end{tikzpicture}
\end{minipage}
\begin{minipage}{.45\linewidth}
\hspace{3cm}
\begin{tikzpicture}[scale=0.95]
\draw[ultra thick,solid ,color=black,mark=none] (0,0) -- (1,0); \draw (1,0) node[anchor=west] {$h_0 = 8$};
\draw[ultra thick,dashed,color=blue ,mark=none] (0,1) -- (1,1); \draw (1,1) node[anchor=west] {$h_0 = 4$};
\draw[ultra thick,dotted,color=red  ,mark=none] (0,2) -- (1,2); \draw (1,2) node[anchor=west] {$h_0 = 1$};
\end{tikzpicture}
\end{minipage}
\caption{Effectivity indices for the smoothed fundamental solution and $p=2$.}
\label{figure_smooth_P1_eff}
\end{center}
\end{figure}
\begin{figure}
\begin{center}
\begin{minipage}{.45\linewidth}
\begin{tikzpicture}[scale=0.95]
\begin{axis}
[
	xlabel = {$N_{\rm dofs}$},
	ylabel = {$\eta_\ell/\|u-u_\ell\|_{H^1_\kappa(\Omega)}$},
	xmode = log,
	ymode = log,
	width = \linewidth
]

\plot[ultra thick,solid ,color=black,mark=none] table[x=dof,y expr=\thisrow{est}/\thisrow{err}] {data/fond/P2_mu_1.00_hini_008.txt};
\plot[ultra thick,dashed,color=blue ,mark=none] table[x=dof,y expr=\thisrow{est}/\thisrow{err}] {data/fond/P2_mu_1.00_hini_004.txt};
\plot[ultra thick,dotted,color=red  ,mark=none] table[x=dof,y expr=\thisrow{est}/\thisrow{err}] {data/fond/P2_mu_1.00_hini_001.txt};

\node[anchor=north east] at (rel axis cs: 0.9,0.9) {$\kappa^2 = 1$};

\end{axis}
\end{tikzpicture}
\end{minipage}
\begin{minipage}{.45\linewidth}
\begin{tikzpicture}[scale=0.95]
\begin{axis}
[
	xlabel = {$N_{\rm dofs}$},
	ylabel = {$\eta_\ell/\|u-u_\ell\|_{H^1_\kappa(\Omega)}$},
	xmode = log,
	ymode = log,
	width = \linewidth
]

\plot[ultra thick,solid ,color=black,mark=none] table[x=dof,y expr=\thisrow{est}/\thisrow{err}] {data/fond/P2_mu_0.10_hini_008.txt};
\plot[ultra thick,dashed,color=blue ,mark=none] table[x=dof,y expr=\thisrow{est}/\thisrow{err}] {data/fond/P2_mu_0.10_hini_004.txt};
\plot[ultra thick,dotted,color=red  ,mark=none] table[x=dof,y expr=\thisrow{est}/\thisrow{err}] {data/fond/P2_mu_0.10_hini_001.txt};

\node[anchor=north east] at (rel axis cs: 0.9,0.9) {$\kappa^2 = 0.1$};

\end{axis}
\end{tikzpicture}
\end{minipage}

\begin{minipage}{.45\linewidth}
\begin{tikzpicture}[scale=0.95]
\begin{axis}
[
	xlabel = {$N_{\rm dofs}$},
	ylabel = {$\eta_\ell/\|u-u_\ell\|_{H^1_\kappa(\Omega)}$},
	xmode = log,
	ymode = log,
	width = \linewidth
]

\plot[ultra thick,solid ,color=black,mark=none] table[x=dof,y expr=\thisrow{est}/\thisrow{err}] {data/fond/P2_mu_0.01_hini_008.txt};
\plot[ultra thick,dashed,color=blue ,mark=none] table[x=dof,y expr=\thisrow{est}/\thisrow{err}] {data/fond/P2_mu_0.01_hini_004.txt};
\plot[ultra thick,dotted,color=red  ,mark=none] table[x=dof,y expr=\thisrow{est}/\thisrow{err}] {data/fond/P2_mu_0.01_hini_001.txt};

\node[anchor=north east] at (rel axis cs: 0.9,0.9) {$\kappa^2 = 0.01$};

\end{axis}
\end{tikzpicture}
\end{minipage}
\begin{minipage}{.45\linewidth}
\hspace{3cm}
\begin{tikzpicture}[scale=0.95]
\draw[ultra thick,solid ,color=black,mark=none] (0,0) -- (1,0); \draw (1,0) node[anchor=west] {$h_0 = 8$};
\draw[ultra thick,dashed,color=blue ,mark=none] (0,1) -- (1,1); \draw (1,1) node[anchor=west] {$h_0 = 4$};
\draw[ultra thick,dotted,color=red  ,mark=none] (0,2) -- (1,2); \draw (1,2) node[anchor=west] {$h_0 = 1$};
\end{tikzpicture}
\end{minipage}
\caption{Effectivity indices for the smoothed fundamental solution and $p=3$.}
\label{figure_smooth_P2_eff}
\end{center}
\end{figure}

\subsection{Singular unknown solution}

Here, we consider the infinite $L$-shape domain
$\Omega = \{ x \in \mathbb{R}^2 \; | \; x_1 \geq 0 \text{ or } x_2 \geq 0 \}$.
The reaction term $\kappa^2$ takes the value $10$ if $x_2 > x_1$, and $0.1$ otherwise,
while $f$ is the indicator function of $(0,1)^2$. This is illustrated in Figure~\ref{figure_setup_sing}. 
Note that in this example $\kappa$ varies ``up to infinity'' and the boundary is unbounded,
so that employing a coupling with a boundary element method is highly non-trivial.

The solution is not smooth and unknown, so that we only present the values
of the estimator $\eta_\ell$ rather than the error itself. Convergence
histories are displayed for $p=1$ and $4$ in Figure~\ref{figure_conv_sing}.
In both cases, the initial mesh is selected with $h_0 = 1$ and the initial mesh of 
active elements consists of the $4$ triangles inside the support of $f$. We observe
optimal convergence rates as expected.

We finally display in Figure~\ref{figure_sing_mesh} meshes generated by the
adaptive algorithm. We first note that localized mesh grading occurs
exactly where expected, namely at the re-entrant corner of the boundary
for $p=1$, and additionally on the remaining corners of the support of $f$
for $p=3$, where higher-order singularities are located. We further observe that
the artificial boundary $\Gamma_H$ is pushed outward more rapidly in the lower part
of the domain than in the upper part, which is in accordance with the different values of
$\kappa$.

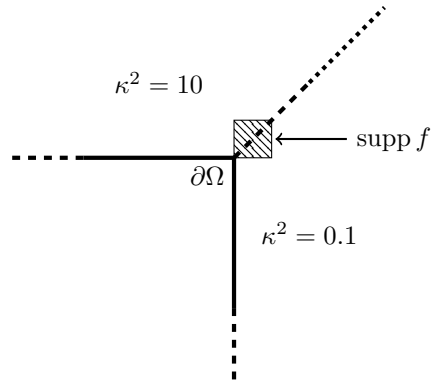
\begin{figure}
\begin{center}
\begin{tikzpicture}

\draw[pattern=north west lines] (0,0) rectangle (0.5,0.5);

\draw[ultra thick,dashed] ( 0,-2) -- ( 0,-3);
\draw[ultra thick,dashed] (-2, 0) -- (-3, 0);
\draw[ultra thick,solid ] ( 0, 0) -- ( 0,-2);
\draw[ultra thick,solid ] ( 0, 0) -- (-2, 0);
\draw[ultra thick,dashed] ( 0, 0) -- ( 1, 1);
\draw[ultra thick,dotted] ( 1, 1) -- ( 2, 2);

\draw (0,0) node[anchor=north east] {$\partial \Omega$};

\draw(-1, 1) node {$\kappa^2 = 10$};
\draw( 1,-1) node {$\kappa^2 = 0.1$};

\draw[thick,<-] (0.55,0.25) -- (1.5,0.25) node[anchor=west] {$\supp f$};
\end{tikzpicture}
\caption{Setup of experiment for the singular unknown solution.}
\label{figure_setup_sing}
\end{center}
\end{figure}

\begin{figure}
\begin{center}
\begin{minipage}{.45\linewidth}
\begin{tikzpicture}[scale=0.95]
\begin{axis}
[
	xlabel={$N_{\rm dofs}$},
	ylabel={$\eta_\ell$},
	xmode=log,
	ymode=log,
	width=\linewidth
]

\plot[ultra thick,black,solid ,mark=none] table[x=dof,y=est] {data/sing/P0.txt};
% \plot[ultra thick,black,dashed,mark=none] table[x=dof,y=est] {data/sing/P0_unif.txt};

\node[anchor=south west] at (rel axis cs: 0.1,0.1) {$p = 1$};

\plot[domain=1.e4:5.e6,dashed] {2.*x^(-0.5)};
\SlopeTriangle{0.6}{-0.15}{0.2}{-0.5}{$N_{\rm dofs}^{-1/2}$}{}
% \plot[domain=1.e4:5.e6,dashed] {2*x^(-.333333)};
% \SlopeTriangle{0.8}{-0.1}{0.35}{-.33}{$N_{\rm dofs}^{-1/3}$}{}

\end{axis}
\end{tikzpicture}
\end{minipage}
\begin{minipage}{.45\linewidth}
\begin{tikzpicture}[scale=0.95]
\begin{axis}
[
	xlabel={$N_{\rm dofs}$},
	ylabel={$\eta_\ell$},
	xmode=log,
	ymode=log,
	width=\linewidth
]

% \plot[ultra thick,black,dashed,mark=none] table[x=dof,y=est] {data/sing/P3_unif.txt};
\plot[ultra thick,black,solid ,mark=none] table[x=dof,y=est] {data/sing/P3.txt};

\node[anchor=south west] at (rel axis cs: 0.1,0.1) {$p = 4$};

\plot[domain=1.e4:5.e4,dashed] {300000*x^(-2.)};
\SlopeTriangle{0.7}{-0.1}{0.2}{-2}{$N_{\rm dofs}^{-2}$}{}
% \plot[domain=2.e3:5.e4,dashed] {2*x^(-.333333)};
% \SlopeTriangle{0.8}{-0.1}{0.65}{-.33}{$N_{\rm dofs}^{-1/3}$}{}

\end{axis}
\end{tikzpicture}
\end{minipage}
\caption{Convergence history for the singular unknown solution.}
\label{figure_conv_sing}
\end{center}
\end{figure}
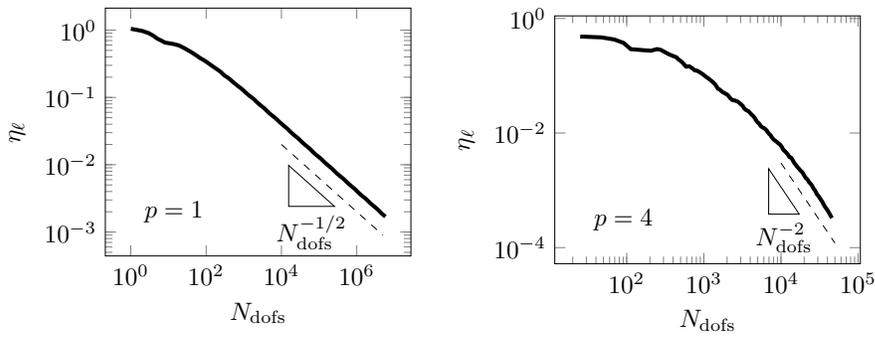

\begin{figure}
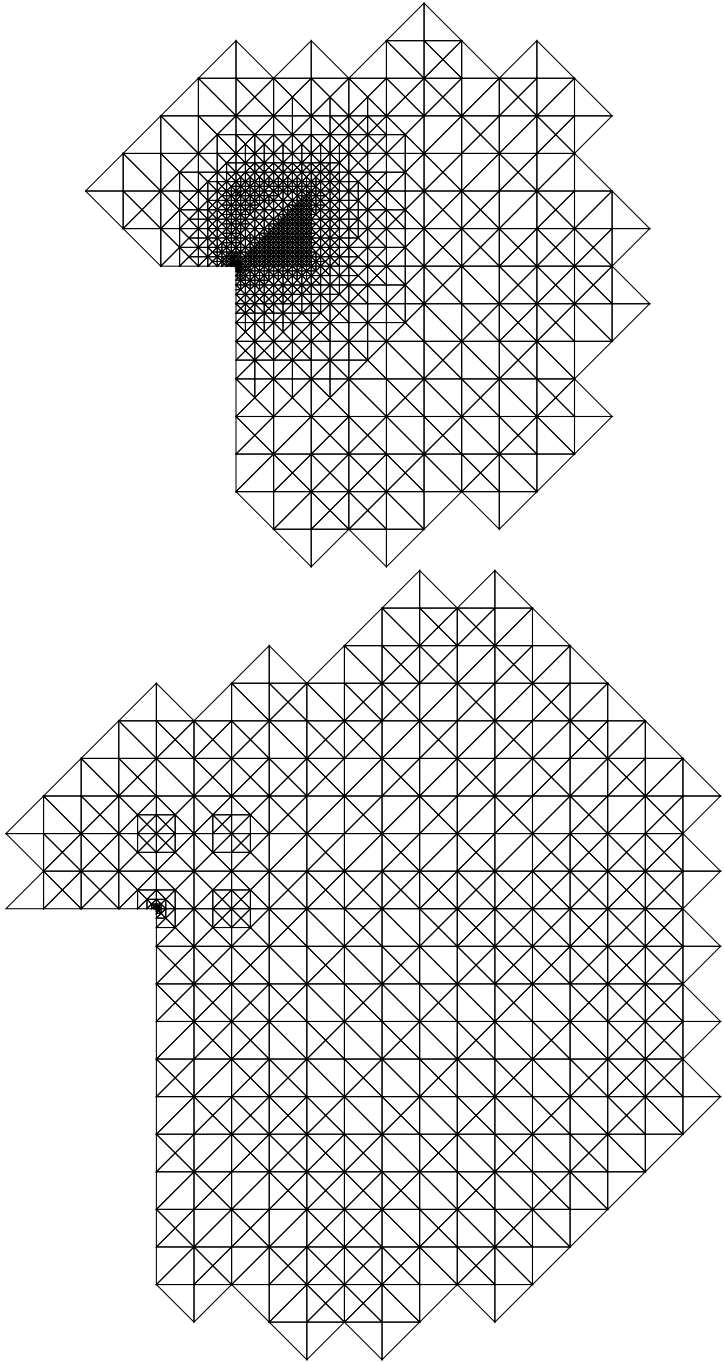

\begin{center}
%\begin{minipage}{.45\linewidth}
\begin{tikzpicture}
\input{data/sing/mesh_P0_30}
\end{tikzpicture}
%\end{minipage}
%\begin{minipage}{.45\linewidth}
\begin{tikzpicture}
\input{data/sing/mesh_P3_50}
\end{tikzpicture}
%\end{minipage}
\caption{Meshes generated for the singular unknown solution and $p=1$ at iteration
$\ell=30$ (top) and $p=4$ at iteration $\ell=50$ (bottom).}
\label{figure_sing_mesh}
\end{center}
\end{figure}

% Bibliography
\bibliographystyle{siamplain}
\bibliography{literature}

\end{document}